\documentclass[12pt,a4paper]{article}  
\usepackage{hyperref}
\usepackage{amsthm}
\usepackage{makeidx}
\usepackage{pifont}
\usepackage{psfrag}
\usepackage{tikz} 
\usepackage{tikz-cd}
  \usetikzlibrary{matrix,arrows}
\usepackage{amsmath,amssymb,amsfonts,graphics}
\usepackage{latexsym,epsf,color,colordvi,epic,amscd}
\usepackage[mathscr]{eucal}
\usepackage{stackrel}

\theoremstyle {definition}
\newtheorem{thm}{Theorem}[section]

\newtheorem{Lemma}[thm]{Lemma}
\newtheorem{Proposition}[thm]{Proposition}
\newtheorem{cor}[thm]{Corollary}

\theoremstyle{definition}
\newtheorem{Remark}[thm]{Remark}
\newtheorem{Example}[thm]{Example} 

\newtheorem{Definition}[thm]{Definition}

\def\acala            {{}_\cala \cala}
\def\acalm            {{}_\cala \calm}

\def\acalaa           {{}_\cala \cala_\cala}

\def\Aop              {{A^\opp_{\phantom i}}}
\def\Awee             {{A^\wee_{\phantom i}}}

\def\be               {\begin{equation}}
\def\bearl            {\begin{array}{l}}
\def\bearll           {\begin{array}{ll}}

\def\Bimod            {\text{-bimod}}
\def\BimoD            {\text{-bimod-}}
\def\boti             {\,{\boxtimes}\,}
\def\Bop              {{B^\opp_{\phantom i}}}
\def\cala             {{\mathcal A}}
\def\Cala             {{\!\mathcal A}}
\def\calaopp          {{\mathcal A^\opp_{\phantom|}}} 
\def\Calaopp          {{\!\mathcal A^\opp_{\phantom|}}}
\def\calb             {{\mathcal B}}
\def\calbopp          {{\mathcal B^\opp_{\phantom|}}}
\def\calc             {{\mathcal C}}
\def\calcopp          {{\mathcal C^\opp_{\phantom|}}}
\def\cald             {{\mathcal D}}
\def\caldopp          {{\mathcal D^\opp_{\phantom|}}}

\def\calm             {{\mathcal M}}

\def\calma            {{\mathcal M}_{\!\cala}}

\def\calmopp          {{\mathcal M^\opp_{\phantom|}}}

\def\caln             {{\mathcal N}}
\def\calna            {{\mathcal N}_{\!\cala}}

\def\calx             {{\mathcal X}}

\def\cir              {\,{\circ}\,}
\def\circtensor       {\setbox0\hbox{\large$\circlearrowleft$} \rlap{\hbox to\wd0{\hss$\times$\hss}}\box0}
\def\circtensorsmall  {\setbox0\hbox{$\circlearrowleft$} \rlap{\hbox to\wd0{\hss$\times$\hss}}\box0}
\def\coev             {\beta}
\def\coevl            {\mathrm{coev}^{\rm l}}
\def\coevr            {\mathrm{coev}^{\rm r}}
\def\delt             {convolution}   
\def\dsty             {\displaystyle }
\def\ee               {\end{equation}}
\def\eear             {\end{array}}

\def\eq               {\,{=}\,}
\def\End              {{\rm End}}
\def\ev               {\delta}
\def\evl              {\mathrm{ev}^{\rm l}}
\def\evr              {\mathrm{ev}^{\rm r}}
\def\findim           {fini\-te-di\-men\-si\-o\-nal}

\def\Fun              {{\mathcal Fun}}
\def\Funle            {{\mathcal Lex}}

\def\Funre            {{\mathcal Rex}}

\def\Gammalr          {\Gamma^{\rm lr}}
\def\Gammarl          {\Gamma^{\rm rl}}

\newcommand\Ho[3]     {{}_{#1}^{}\Langle{#2}\,,{#3}\Rangle}
\newcommand\HO[3]     {{}_{#1}^{\phantom|}\big\Langle{#2}\,,{#3}\big\Rangle}
\def\Hom              {{\rm Hom}}

\newcommand\Holr[4]   {{}_{#1|#2}^{}\Langle{#3}\,,{#4}\Rangle}
\newcommand\Hop[3]    {{}_{#1}\Langle{{#2}}\,,{#3}\Rangle}
\newcommand\Hox[3]    {{}_{#1}^{}\Langle{#2}\,,{#3}\Rangle_{}^\wee}

\def\id               {{\rm id}}
\def\IHom             {\underline{\rm Hom}}

\def\iN               {\,{\in}\,}

\def\ko               {{\ensuremath{\Bbbk}}}
\def\la               {{\rm l.a.}}
\def\Langle           {\langle }
\def\lla              {{\rm l.l.a.}}

\def\llMrr            {{}^{{\rm ll}\!\!}\calm^{{\rm rr}}}
\def\Mod              {\text{-mod}}
\def\moD              {\text{mod-}}
\def\mwee             {{m^\wee_{\phantom i}}}

\def\Nat              {\mathrm{Nat}}

\newcommand\Nxl[1]    {\\[-1.3em]\\[#1mm]}

\def\ol               {\overline }
\def\one              {{\bf1}}

\def\opp              {\mathrm{opp}}
\def\otA              {\,{\otimes_{\!A}}\,}
\def\otB              {\,{\otimes_{\!B}}\,}
\def\oti              {\,{\otimes}\,}
\def\otik             {\,{\otimes_\ko}\,}
\def\pift             {{\rm N}^{\rm r}}
\def\pifu             {{\rm N}^{\rm l}}
\def\Phile            {\Phi^{\rm l}}

\def\Phire            {\Phi^{\rm r}}
\def\Psile            {\Psi^{\rm l}}

\def\Psire            {\Psi^{\rm r}}
\def\ra               {{\rm r.a.}}

\def\Rangle           {\rangle }
\def\rra              {{\rm r.r.a.}}

\def\rrMll            {{}^{{\rm rr}\!}\calm^{{\rm ll}}}
\def\rrNll            {{}^{{\rm rr}\!}\caln^{{\rm ll}}}
\def\Sl               {{\rm S}^{\rm l}}
\def\Sr               {{\rm S}^{\rm r}}

\def\sss              {\scriptscriptstyle}

\def\Times            {\,{\times}\,}
\def\To               {\,{\to}\,}

\def\vect             {\ensuremath{\mathrm{vect}}}
\def\Vee              {{}^{\vee\!}}
\newcommand\void[1]   {}
\def\wee              {*}  

\begin{document}

\numberwithin{equation}{section}

\thispagestyle{empty}
\begin{flushright}
   {\sf ZMP-HH/16-26}\\
   {\sf Hamburger$\;$Beitr\"age$\;$zur$\;$Mathematik$\;$Nr.$\;$630}\\[2mm]
\end{flushright}
\vskip 2.0em

\begin{center}{\bf \Large 
Eilenberg-Watts calculus for finite categories
\\[6pt]
and a bimodule Radford {\boldmath $S^4$} theorem}

\vskip 18mm

 {\large \  \ J\"urgen Fuchs\,$^{\,a,c}, \quad$ Gregor Schaumann\,$^{\,b}, \quad$
 Christoph Schweigert\,$^{\,c}$
 }

 \vskip 12mm

 \it$^a$
 Teoretisk fysik, \ Karlstads Universitet\\
 Universitetsgatan 21, \ S\,--\,651\,88\, Karlstad \\[9pt]
 \it$^b$
 Fakult\"at f\"ur Mathematik, \
 Universit\"at Wien\\
 Oskar-Morgenstern-Platz 1,
 \ A\,--\,10\,90 
   Wien
 \\[9pt]
 \it$^c$
 Fachbereich Mathematik, \ Universit\"at Hamburg\\
 Bereich Algebra und Zahlentheorie\\
 Bundesstra\ss e 55, \ D\,--\,20\,146\, Hamburg

\end{center}

\vskip 3.2em

\noindent{\sc Abstract}\\[3pt]
We obtain Morita invariant versions of Eilenberg-Watts type theorems, relating
Deligne products of finite linear categories to categories of left exact as well as 
of right exact functors. 
This makes it possible to switch between different functor categories as well as 
Deligne products, which is often very convenient. For instance, 
we can show that applying the equivalence from left exact to right exact functors to
the identity functor, regarded as a left exact functor, gives a Nakayama functor.
The equivalences of categories we
exhibit are compatible with the structure of module categories over
finite tensor categories. This leads to a generalization of Radford's
$S^4$-theorem to bimodule categories. We also explain the relation of our
construction to relative Serre functors on module categories
that are constructed via inner Hom functors.



\setcounter{footnote}{0} \def\thefootnote{\arabic{footnote}} 

\vskip 4em

 \newpage

\section{Introduction}

A classical result in algebra, the Eilenberg-Watts theorem \cite{eile4,wattC}, states
that, given two unital rings $R$ and $S$, any right exact functor $F\colon R\Mod\To S\Mod$ 
that preserves small coproducts is naturally isomorphic to tensoring with an $S$-$R$-bimodule.
Further, this bimodule can be expressed explicitly through the functor $F$, 
namely as the left $S$-module $F(_RR)$ endowed with a natural right $R$-action.

While this formulation of the statement is extremely useful, it hides the categorical 
nature of the situation: For a category that can be realized as the category of
modules over a ring, that ring is determined only up to Morita equivalence.
But nevertheless the bimodule in the theorem appears to make explicit use of the
choice of ring. This can pose a problem in situations in which the
category is given in more abstract terms and a Morita invariant formulation is desired.

Moreover, it is known that similar results exist for left exact functors, expressing them, 
under certain finiteness conditions, in terms of a Hom functor (see e.g.\ \cite{ivanS})
applied to a bimodule. When taken together, for sufficiently nice categories
one thus deals with two categories -- of left exact functors and of right exact functors, 
respectively -- that should both be related to suitable categories of bimodules.

In the present paper, we place ourselves under the following finiteness conditions,
which in particular allows us to rely on results from \cite{etno2,shimi7}
and to relate them to the Eilenberg-Watts equivalences:
We fix an algebraically closed field \ko\ and work with finite \ko-linear categories, i.e.\
categories that are equivalent to categories of finite-dimensional
modules over a finite-dimensional \ko-algebra.
This ensures the existence of various categorical constructions, in particular of the 
Deligne product and of certain ends and coends. 

One might wonder to what extent the results of this article can be extended beyond 
finite categories.  For cocomplete locally presentable categories, for example,
there exists still a product $\boti$ that is universal with respect to cocontinuous functors.
However, most of the constructions of this article fail in this bigger category:  
in \cite[Thm 1.4]{brcJ} examples for such categories $\calc$ are given, where 
the categories $\calc \boti \Hom(\calc,\vect)$ and $\End(\calc)$ are not equivalent 
(here $\vect$ denotes the category of all vector spaces and $\Hom$ and $\End$ refer to 
the categories of cocontinuous functors).

\medskip

Besides giving a Morita invariant formulation of the Eilenberg-Watts equivalences, this paper 
connects these equivalences with other representation theoretic concepts, in particular
with Nakayama functors and, for the case of rigid monoidal categories, with Radford's theorem 
on the fourth power of the antipode of a Hopf algebra. 

To give a more detailed account of our findings, let us first recall that,
given two finite linear categories $\cala \eq A\Mod$ and $\calb \eq B\Mod$, the 
opposite category $\calaopp$ can be identified with the category of
right $A$-modules and the Deligne product with $B$-$A$-bi\-modules,
  \be
  \calb \boti \calaopp \cong B\BimoD A \,.
  \label{i:BbimodA}
  \ee
A first result towards a categorical formulation is 
a categorical variant of the Peter-Weyl theorem: We can express the regular $A$-bimodule
$A \eq {}_AA_A$ and the co-regular $A$-bimodule $A^* \eq {}^{}_AA_A^*$ as an end and as a 
coend, respectively, of a functor with values in $A\Bimod$; specifically, we have
  \be
  A \ = \int_{m\in A\Mod} m\boxtimes m^\wee \qquad\text{and}\qquad
  A^\wee \,= \int^{m\in A\Mod}\! m\boxtimes m^\wee
  \label{i:PW}
  \ee
as $A$-bimodules.
More generally, we show in Proposition \ref{prop:end(GotikWee} that for any
$\ko$-linear functor $G\colon A\Mod \To B\Mod$ the functor
from $A\Mod \Times A\Mod^\opp$ to $B\BimoD A$ that is defined by 
$ m \Times \ol n \,{\longmapsto}\, G(m) \,{\otik}\, n^\wee $ has as an end the 
$B$-$A$-bimodule
  \be
  \int_{m\in A\Mod} G(m) \otik m^\wee =\, G(A) \,,
  \ee
and as a coend the $B$-$A$-bimodule
  \be
  \int^{m\in A\Mod}\!\! G(m) \otik m^\wee = G(A^\wee_{})
  \ee
(with structure morphisms as given in \eqref{eq:i^A_m} and \eqref{eq:i^A*_m}, respectively).

With the help of this result we can set up a triangle 
  \be
  \begin{tikzcd}[row sep=11ex]
  ~ & \calaopp \boti \calb ~ \ar{dl}[xshift=-2pt]{\Phile} \ar[xshift=-2pt]{dr}[swap]{\Phire}
  & ~ \\
  \Funle(\cala,\calb) \ar[yshift=3pt]{rr}[yshift=1.3pt]{\Gammarl} \ar[xshift=-12pt]{ur}[xshift=2pt]{\Psile}
  & ~ & \Funre(\cala,\calb) \ar[yshift=-3pt]{ll}{\Gammalr} \ar[xshift=12pt]{ul}[swap]{\Psire}
  \end{tikzcd}
  \label{i:AoppB-lex-rex}
  \ee
of Eilenberg-Watts type equivalences of finite linear categories.
Here $\Funle(-,-)$ is the category of left exact functors and $\Funre(-,-)$ the one of right
exact functors. The Deligne product $\calb \boti \calaopp \,{\cong}\, \calaopp \boti \calb$
plays the role of the category of bimodules, as in the relation \eqref{i:BbimodA} above.
The Eilenberg-Watts equivalences for categories of left exact
functors have appeared in \cite[Rem.\,2.2(i)]{etno2} and \cite[Section\,3.4]{shimi7};
they are
  \be
  \begin{array}{rrcl}
  \Phile_{} : & \calaopp \boti \calb &\!\! \longrightarrow \!\!& \Funle(\cala,\calb)
  \Nxl1
  & \ol a \boti b &\!\! \longmapsto \!\!& \Hom_\cala (a,-) \oti b 
  \Nxl4
  \text{and} \qquad
  \Psile_{} : & \Funle(\cala,\calb) &\!\! \longrightarrow \!\!& \calaopp \boti \calb
  \Nxl1
  & F &\!\! \longmapsto \!\!& \int^{a\in\cala} \ol a \boti F(a) \,.
  \Nxl4
  \eear
  \label{i:phipsi-l}
  \ee
That these functors are quasi-inverses is essentially a consequence of the Yoneda
lemma. The case of right exact functors uses instead of the Hom-pairing $\Hom_\cala$
(which constitutes a left exact functor $\calaopp\boti\cala\To\vect$) 
a pairing that is familiar from the study of Serre functors: the right exact pairing
$\Hom_\cala(-,-)^\wee_{}$ obtained by using the vector space dual.
(Recall that the morphism spaces of the categories under consideration are \findim.)
Explicitly, we show that the following pair of functors are quasi-inverse equivalences:
  \be
  \begin{array}{rrcl}
  \Phire_{} : & \calaopp \boti \calb &\!\! \longrightarrow \!\!& \Funre(\cala,\calb)
  \Nxl1
  & \ol a \boti b &\!\! \longmapsto \!\!& \Hom_\cala( -,a)^\wee_{} \oti b
  \Nxl4
  \text{and} \qquad
  \Psire_{} : & \Funre(\cala,\calb) &\!\! \longrightarrow \!\!& \calaopp \boti \calb
  \Nxl1
  & G &\!\! \longmapsto\!\! & \int_{a\in\cala} \ol a\, \boti G(a) \,.
  \eear
  \label{i:phipsi-r}
  \ee
We note that these results provide a convenient way to think about Deligne products 
in terms of left or right exact functors, respectively. Regarding $\calaopp \boti \calb$
as a categorification of matrix elements, we describe functors in terms of matrix elements 
and obtain a categorified matrix calculus.
This should prove useful in various applications, e.g.\ in situations in which
relative tensor products and relative (twisted) centers play a role, like in
topological field theories. Specifically, one can generalize constructions
established for semisimple categories that involve sums over simple objects to the
non-semisimple case by using instead ends or coends. This can e.g.\ be a key step 
in the passage to non-semisimple variants of modular functors.

Composing the equivalences \eqref{i:phipsi-l} and \eqref{i:phipsi-r} leads to natural 
equivalences $\Gammalr$ and $\Gammarl$ between the categories of left exact and right 
exact functors. We can then in particular consider the right exact endofunctor
  \be
  \pift_\calx := \Gammarl(\id_\calx)
  \,= \int^{x\in\calx}\!\! \Hom_\calx (-,x)^* \oti x \,.
  \label{i:pift}
  \ee
For $\calx \eq A\Mod$ a category of modules, by the Peter-Weyl formula \eqref{i:PW} this 
turns out to be the Na\-ka\-yama functor 
  \be
  \pift_{A\Mod} = A^\wee\otimes_A- \,\cong\, \Hom_{A\Mod}(-,A)^\wee_{} .
  \ee
In other words, we have arrived at a Morita invariant description of the Nakayama functor.
It follows immediately from the triangle \eqref{i:AoppB-lex-rex} of
equivalences that there is also a left exact analogue of the Nakayama functor, namely
  \be
  \pifu_\calx := \Gammalr(\id_\calx)
  \,= \int_{x\in\calx} \Hom_\calx (x,-) \oti x \,.
  \label{i:pifu}
  \ee
In the case $\calx \eq A\Mod$ of modules over an algebra the so defined left exact functor is 
given by $\pifu_{A\Mod} \,{=}\, \Hom_{\moD A}(-^\wee,A) \,{\cong}\, \Hom_{A\Mod}(A^\wee,-)$.

In the Radford theorem iterated duals enter crucially. Similarly, the
following property of the Nakayama functors \eqref{i:pift} and \eqref{i:pifu},
proven in Theorem \ref{thm:picu.F=Fll.pifu}, turns out to be essential: 
For $F\colon \cala\To\calb$ a left exact functor for which the left adjoint $ F^\lla $ of 
its left adjoint $F^\la$ exists and is again left exact, there is a natural isomorphism
  \be
  \pifu_\calb \circ F \,\cong\, F^\lla \circ \pifu_\Cala
  \ee
of functors. An analogous result holds for right exact functors.  

\medskip

In the second part of this paper we consider the particular case that the finite
$\ko$-linear category $\calm$ has the additional structure of a bimodule category over 
finite tensor categories $\cala$ and $\calb$. We show, in Theorem \ref{thm:pifuMbimodfunc}, 
that in this case the left exact Nakayama functor $\pifu_\calm$ of the 
category $\calm$ has a canonical structure of a bimodule 
functor
  \be
  \pifu_\calm\colon~ \calm \to \rrMll ,
  \ee
that is, there are coherent isomorphisms
  \be
  \pifu_\calm(a.m.b) \cong a^{\vee\vee\!}.\,\pifu_\calm(m)\,.\,{}^{\!\vee\vee\!}b
  \ee
for all $m\iN\calm$, $a\iN\cala$ and $b\iN\calb$. Here $a^\vee$ is the right dual and 
${}^{\vee\!} a$ the left dual of $a\iN\cala$. Similarly (Theorem \ref{thm:piftMbimodfunc})
for the Nakayama functor there are coherent isomorphisms
  \be
  \pift_\calm(a.m.b) \cong {}^{\vee\vee\!}a.\,\pift_\calm(m)\,.b^{\vee\vee}
  \ee

These statements actually imply and generalize Radford's classical $S^4$-theorem for 
Hopf algebras; indeed, as shown in Lemma \ref{lem:pi=D...}, for $\cala$ regarded as a 
bimodule category over itself the Nakayama functors are, up to dualities,  given by 
tensoring with the distinguished invertible object
$D \iN \cala$. It follows that the quadruple dual of $\cala$ satisfies
  \be
  {-}^{\vee\vee\vee\vee} \,\cong\, D \oti {-} \oti D^{-1} .
  \ee
This is just the categorical formulation of Radford's theorem as obtained in \cite[Thm.\,3.3]{etno2}. 
   %
   Let us mention that the latter also has a topological explanation, within the framework
   of framed topological field theory: it can be obtained via the so-called belt trick, which is
   based on the fact that the fundamental group of $\mathrm{SO}(3)$ has order 2
   \cite[Sect.\,4.3]{doSs3}. In that framework, our generalized version of Radford's theorem
   would correspond to performing the
   belt trick on a belt with an embedded defect line that is labeled by the bimodule category $\calm$.
   Note, however, that in our algebraic approach the definition of the Nakayama functor is
   naturally given in the setting of finite {\em linear} categories, without assuming a 
   bimodule structure on $\calm$.

Another direct application of our results, pointed out to us by Shimizu \cite{shimiP}, is
Theorem \ref{thm:shimiP}: A finite multitensor category is equivalent as a linear category
to the category of modules over a symmetric Frobenius algebra if and only if it
is unimodular and its double dual functor is isomorphic to the identity functor.

             \medskip

Nakayama functors are closely related with Serre functors. 
Recall that a right Serre functor on a linear Hom-finite additive category $\calx$ is
an additive endofunctor $G$ together with a natural family of isomorphisms between the morphism
spaces $\Hom_\calx(x,y)$ and $\Hom_\calx (y,G(x))^\wee$; left Serre functors are defined 
analogously. In our setting, a Serre functor exists if and only if $\calx$ is semisimple.
On the other hand, if $\calx \eq \calm$ is a module category, the {\em inner Hom functor}
$\IHom(-,-)$ allows one to define the notion of \emph{relative Serre functors} $\Sr_\calm$ and
$\Sl_\calm$ (see also \cite[Sect.\,4.4]{schaum5}) as functors that instead come with families
  \be
  \IHom(m,n)^\vee \xrightarrow{~\cong~}\, \IHom(n,\Sr_\calm(m))
  \qquad \text{and} \qquad
  {}^{\vee}\IHom(m,n) \,\xrightarrow{~\cong~}\, \IHom(\Sl_\calm(n),m)
  \ee
of isomorphisms, natural in $m,n\iN\calm$, for a right and a left relative Serre functor, 
respectively. A necessary and sufficient condition for the existence of right and left 
relative Serre functors for a module category $\calm$ is that $\calm$ is an exact module 
category (Proposition \ref{prop:exirelSerre}).

Finally we show in Theorem \ref{thm:N=D.S} that for an exact module category $\calm$
over a finite tensor category $\cala$ with distinguished invertible object $D_\cala$
there is an equivalence
  \be
  \pifu_\calm \,\cong\, D_\cala\,.\,\Sl_\calm
  \ee
between the Nakayama functor and the left relative Serre functor composed with the
action by the distinguished element.
This result suggests in particular that an exact module category $\calm$ should be called 
unimodular iff there exists a module natural isomorphism between the Nakayama functor and the
right relative Serre functor of $\calm$. Then in particular the finite tensor category $\cala$ 
is unimodular iff the regular $\cala$-module category $\acala$ is unimodular. 

 \medskip

To conclude this survey of results, let us stress a major virtue
of our categorical formulation of the Eilenberg-Watts theorem: it allows us, for finite
linear categories, to switch back and forth between Deligne products of categories and 
categories of right exact as well as of left exact functors 
and thereby also between those two types of functor categories. 
As a consequence we can understand features of Deligne products in terms of functor
categories and, conversely, aspects of functors in terms of a categorified matrix calculus;
for instance, certain ends and coends can be naturally interpreted as providing a
categorified matrix multiplication (see Corollary \ref{cor:int=G'G-etc}).
We take the fact that the structures which arise when setting up an Eilenberg-Watts calculus 
for a (finite linear) category $\calc$ fit well with the additional structure of a module
category on $\calc$ as a further indication that the structures investigated in this paper
are natural and can be of much avail.


\section{Preliminary results}

\subsection{Notation and background}\label{ssec:background}

Throughout the paper we make the following assumptions. We work over a fixed field \ko;
all categories are \ko-linear and all functors \ko-linear or bilinear. All \ko-algebras
and all modules and bimodules over them are \findim. For such (unital associative) 
algebras $A$ and $B$ we denote by $A\Mod$, $\moD A$ and $B\BimoD A$ the categories 
of \findim\ left and right $A$ modules and of \findim\ $B$-$A$-bimodules, respectively, and
we write $\Hom_A(-,-)$ for morphisms of left $A$ modules and $\Hom_{B|A}(-,-)$ for
morphisms of $B$-$A$-bimodules. Two distinguished $A$-bimodules are the regular bimodule
${}_AA_A$ with the actions given by the product of $A$ and the co-regular bimodule
${}_A^{}A_A^\wee$ for which the actions are obtained by dualizing those for ${}_AA_A$.

Furthermore, finiteness properties of the categories involved are essential for our work:
we require, unless specified otherwise, that categories are in addition \emph{finite}. As 
in \cite{etos} we assume that \ko\ is algebraically closed; a \ko-linear category 
is finite iff \cite[Sect.\,2.1]{etos} it is equivalent as a linear category 
to the category of \findim\ (left or right) modules over some \findim\ \ko-algebra $A$. 
Our main interest lies in abstract finite linear categories and Morita invariant statements
rather than in statements involving a concrete choice of algebra $A$. But we do have to make
use of pertinent information about \findim\ algebras.
 
We start by recalling variants (see e.g.\ \cite[Thms.\,2.4,\,2.6,\,2.7]{ivanS} and
\cite[Lemma\,2.4]{shimi7}) of the Eilenberg-Watts theorem.
Denote the monoidal category of \findim\ vector spaces over \ko\ by $\vect$, and
for $v \iN \vect$ the dual vector space $\Hom_\vect (v,\ko)$ by $v^\wee$.
Then we have the following standard results:

\begin{Lemma}\label{lem:shimi7:2.6}
Let $A$ and $B$ be \findim\ \ko-algebras and $F, G \colon A\Mod \To B\Mod$ \ko-linear 
functors. Then the following statements are equivalent:
\\[2pt]
(L1)\, $F$ is left exact.
\\[1pt]
(L2)\, $F$ admits a left adjoint.
\\[1pt]
(L3)\, $F \,{\cong}\, \Hom_A(M,-)$ with $M \eq {(F({}_A^{} A _A^\wee))}^\wee_{}$ 
(thus in particular $F$ is representable).
\\[3pt]
Likewise, the following are equivalent:
\\[2pt]
(R1)\, $G$ is right exact.
\\[1pt]
(R2)\, $G$ admits a right adjoint.
\\[1pt]
(R3)\, $G \;{\cong}\; G({}_AA_A)\,{\otimes_A}\,{-}$\,.
\\[1pt]
(R4)\, $G \;{\cong}\; \Hom_A(-,G({}_{A}A_{A})^\wee)^\wee$.
\end{Lemma}

Note that a functor that is left or right adjoint to a linear functor is linear as well.
Also, the vector spaces ${(F({}_A^{} A _A^\wee))}^\wee_{}$ in (L3) (and analogously
$G({}_AA_A)$ in (R3) and ${(G({}_A^{} A _A^\wee))}^\wee_{}$ in (R4)) is
endowed with the appropriate natural structure of a $B$-$A$-bimodule, as e.g.\
described explicitly in the proof of Proposition \ref{prop:end(GotikWee} below.
Further, the isomorphism in (R4) is obtained from the one in (R3) with the help of the
description 
  \be
  \Hom_A({}_AX,(M_A)^\wee)^\wee \cong \Hom_\ko(M_A \otA {}_AX,\ko)^{*}
  = (M_A \otA {}_AX)^{\wee\wee} \cong M_A \otA {}_AX 
  \label{eq:R4expl}
  \ee
of the tensor product $M_A \otA {}_{A}X$ of \findim\ modules over an algebra.
Note that $(-)^\wee\colon A\Mod \To \moD A$ is an exact functor because it is an 
equivalence, and while $\Hom_A\colon A\Mod^\opp \Times A\Mod\To \vect$ is left exact, 
$\Hom_A(-,-)^\wee\colon A\Mod \Times A\Mod^\opp \To \vect$ is right exact.

We write $F^\la$ and $F^\ra$ for the left and right adjoint of a functor $F$, respectively 
(provided they exist). For abelian categories $\calc$ and $\cald$ we denote by
$\Funle(\calc,\cald)$ and $\Funre(\calc,\cald)$ the categories of left exact and right
exact linear functors from $\calc$ to $\cald$, respectively.

The calculus of Eilenberg-Watts also leads directly to the following result, which involves
the Nakayama functor that we will discuss in more detail in Section \ref{ssec:Nakayama}. 

\begin{Lemma}\label{lem:lexrexbimod}
(i)\, For $A$ and $B$ \findim\ algebras there is an equivalence between the categories 
$\Funle(A\Mod,B\Mod)$ and $(A\BimoD B)^\opp$ which maps a functor $F$ to 
     ${(F({}_A^{} A _A^\wee))}^\wee_{}$
as well as an equivalence $\Funre(A\Mod, B\Mod) \,{\simeq}\, (A\BimoD B)^\opp$, mapping 
a functor $G$ to the $B$-$A$-bimodule $G({}_{A}A_{A})^\wee$. 
\\[2pt]
(ii)\, For any \findim\ algebra $A$ there is an equivalence 
$\Funle(A\Mod,A\Mod) \,{\xrightarrow{~\simeq~}} 
     $\linebreak[0]$
\Funre(A\Mod,A\Mod)$. This equivalence maps
the identity functor $\id_{A\Mod}$, regarded as left exact functor, to the 
right exact functor $({}_{A}A_{A})^\wee \otimes_{A}-$\,.
\end{Lemma}

\begin{proof}
(i) is an easy consequence of the parts (L3) and (R3), respectively, of Lemma
\ref{lem:shimi7:2.6}.
(ii) follows by composing the two equivalences from (i) in the case $B \eq A$.
\end{proof}
 
By the equivalence of finite linear categories to those of modules over \findim\ \ko-algebras,
Lemma \ref{lem:shimi7:2.6} immediately gives (compare 
\cite[Prop.\,1.7\,\&\,Cor.\,1.10]{doSs})

\begin{cor}\label{cor:leftexact-etc}
Let $F,G\colon \calc \To \cald$ be linear functors between finite linear categories. 
The following statements are equivalent:
\\[2pt]
(L1)\, $F \iN \Funle(\calc,\cald)$.
\\[1pt]
(L2)\, $F$ admits a left adjoint.
\\[4pt]
Likewise, the following are equivalent:
\\[2pt]
(R1)\, $G \iN \Funre(\calc,\cald)$.
\\[1pt]
(R2)\, $G$ admits a right adjoint.
\\[4pt]
Moreover for $\cald \eq \vect$ in addition the following statements are
equivalent to (L1) and (L2) and to (R1) and (R2), respectively:
\\[2pt]
(L3)\, $F$ is representable, i.e.\ $F \,{\cong}\, \Hom_\calc (c,-)$ for some $c\iN\calc$.
\\[2pt]
(R3)\, $G$ is `dually representable', i.e.\ $G \,{\cong}\, \Hom_\calc (-,d)^*_{}$ 
for some $d\iN\calc$.
\end{cor}

In the sequel we will use the abbreviated notation
  \be
  \Hom_\calc(c,d) =:\, \Ho\calc cd 
  \label{eq:Hom-Ho}
  \ee
for morphism spaces.
Apart from its brevity the motivation for this notation is that we like to think informally
of the Hom functor as a categorified inner product. Moreover, if $\calc$ is $\vect$ or a category
of (bi)modules we further abbreviate $\Ho\vect --\,{ =:}\,\, \Ho\ko-- $ and
  \be
  \Ho{A\Mod} -- =: \Ho A-- \,, \qquad {\rm and} \qquad
  \Ho{B\BimoD A} -- =: \Ho {B|A} -- \,, 
  \ee
respectively.

Whenever it does not make expressions too clumsy, we will write $\ol c$ and $\ol f$
for the objects and morphisms in the opposite category $\calcopp$ that correspond to an 
object $c$ and a morphism $f$ in $\calc$. Thus with the notation \eqref{eq:Hom-Ho}
we have $ \Hop\calcopp {\ol d}{\ol c} \eq \Ho\calc cd $.
If $\calc$ is monoidal then so is $\calcopp$, and we take its tensor product to be
given by $\ol c \,{\otimes_\calcopp}\, \ol{c'} \eq \ol{c \,{\otimes_\calc}\, c'}$.
Then the structure of a right duality on $\calc$ induces a left duality on $\calcopp$
and vice versa.
For a functor $F\colon \calc\To\cald$ the \emph{opposite functor} is the functor 
$F^\opp_{}\colon \calcopp\To\caldopp$ that 
consists of the same maps on the class of objects and on the morphism sets as $F$, i.e.\
  \be
  F^\opp_{\phantom|}(\ol c) = \ol{F(c)} \qquad{\rm and}\qquad 
  F^\opp_{\phantom|}\big( \ol{m'}{\xrightarrow{\sss \ol f}}\ol m \big)
  = \ol{F(m{\xrightarrow{\sss f}}m')} \,.
  \ee
If $F$ is right exact, then $F^\opp_{}$ is left exact, and vice versa.

We are accustomed with the distinguished role that the Hom functor, i.e.\ left exact (bi)functor 
$\Ho\calc -- \colon \calcopp\Times\calc \To \vect$, plays in many contexts. As it turns out,
for us the right exact (bi)functor obtained from Hom by taking duals in $\vect$, which is 
familiar from the definition of Nakayama and Serre functors (and which was already used in
formula \eqref{eq:R4expl} and part (R3) of Corollary \ref{cor:leftexact-etc}), is of similar 
importance. We therefore introduce it formally:

\begin{Definition}\label{Def:dualHom}
For $\calc$ a finite linear category, the \emph{dual Hom} functor is the functor
  \be
  \bearll
  \Hox\calc -- \colon &\calcopp\Times\calc \xrightarrow{~~~} \vect
  \Nxl1 & \hspace*{1.1em}
  (\ol c,c') \longmapsto \big( \Ho\calc {c'}c \big)^*_{} .
  \eear
  \ee
\end{Definition}

Just like in the case of Hom, we will use the term dual Hom functor also for the 
right exact functor $\Hox\calc -c \colon \calc\To\vect$
with fixed $c\iN\calc$ and for the contravariant functor $\Hox\calc c-$.  
Finally recall that a linear category is a module category over $\vect$; by definition,
for $c \iN \calc$ and $v \iN \vect$ the object $c \oti v \,{\cong}\, v \oti c \iN \calc$ 
is the one that corepresents the functor $\Ho\calc -c \otik v$, so that 
  \be
  \Ho\calc {c \oti v} {c' \oti v'} \,\cong\, v^\wee \otik \Ho\calc c{c'} \otik v' 
  \,\cong\, \HO\vect v {\Ho\calc c{c'} \otik v'}
  \label{Hom(av,a'v')}
  \ee
for $c,c' \iN \calc$ and $v,v' \iN \vect$.
Note that in $\calcopp$ we have $\ol{v \oti c} \eq v^\wee \oti \ol c$.


\subsection{Ends and coends}

Recall \cite[Ch.\,IX.4]{MAcl} that a dinatural transformation $F\,{\Rightarrow}\,x$ from a 
functor $F\colon\, \calcopp\Times\calc\To\cald$ to an object $x\,{\in}\,\cald$ is a family 
$\varphi \,{=}\, \{ \varphi_c^{}\colon F(\ol c,c)\To x \}_{c\in\calc}^{}$ of morphisms 
satisfying $\varphi_{c'}^{} \cir F(\ol{c'},f) \eq \varphi_c^{} \cir F(\ol f,c)$
for all $f\,{\in}\,\Ho\calc c{c'}$.
A coend $(z,\iota)$ for a functor $F\colon \calcopp\Times\calc\To\cald$
is an object $z\,{\in}\,\cald$ together with a dinatural transformation $\iota$ 
from $F$ to $z$ having the universal property that for any dinatural transformation
$\varphi\colon F\,{\Rightarrow}\,x$ to some $x\,{\in}\,\cald$ there is a unique morphism
$\kappa \eq \kappa(\varphi) \iN \Ho\cald zx$ such that $\varphi_c \,{=}\, \kappa \,{\circ}\, 
\iota_c$ for all $c \iN \calc$. The notion of an end for a functor
$F\colon \calcopp\Times\calc\To\cald$ is defined dually.
We often suppress the universal dinatural transformation and denote the coend and end of a
functor $F\colon \calcopp\Times\calc\To\cald$, as well as the underlying objects, by
$\int^{c\in\calc\!}\! F(\ol c,c)$ and by $\int_{c\in\calc} F(\ol c,c)$, respectively.

Properties of ends and coends will play a crucial role in the formulation of our results. In 
the present and the next subsection we collect the most basic of those properties.
We start by recalling the following fundamental preservation results in which,
unlike in the rest of the paper, $\cala$ is not assumed to be linear or finite.

\begin{Lemma}\label{lem:basic-coend-pres}~\\[2pt]
{\rm(i)}\,
Let $H\colon \calaopp \Times \cala \To \calc$ be a functor whose coend exists. Then there
is a natural isomorphism
  \be
  \HO\calc {\int^{a\in\cala}\! H(\ol a,a)} - \,\cong \int_{a\in\cala} \Ho\calc {H(\ol a,a)}- 
  \label{eq:coend-hom}
  \ee
of functors from $\calc$ to $\vect$.
\\[3pt]
{\rm (ii)}\,
Let $H\colon \calaopp \Times \cala \To \calc$ be a functor whose end exists. Then there
is a natural isomorphism
  \be
  \Ho\calc - {\int_{a\in\cala} H(\ol a,a)} \,\cong \int_{a\in\cala} \Ho\calc -{H(\ol a,a)} \,.
  \label{eq:coend-hom2r}
  \ee
\end{Lemma}

Another result that will often be used and holds in generality is that for any pair of 
functors $F,\,G\colon \cala \To \calb$ there is a bijection \cite[Eq.\,IX.5(2)]{MAcl}
  \be
  \Nat(F,G) \,\cong \int_{a\in\cala} \Ho\calb {F(a)}{G(a)}
  \label{eq:end-nat-trans}
  \ee
between the set of natural transformations and the end on the right hand side. 
The dinatural transformation for this end just consists of writing a natural transformation 
as a tuple in its components.

\medskip

Restricting our attention from now on again to finite linear categories, we know from
Corollary \ref{cor:leftexact-etc} that left exact functors $F\colon \calb \To \calc$ are 
representable and thus continuous, while right exact functors $G\colon \calb \To \calc$ are
cocontinuous.
Now the end can be written as a (small) limit and thus commutes with continuous functors, 
while the coend commutes with cocontinuous functors. Thus in particular we have

\begin{Lemma}
Let $\cala$, $\calb$ and $\calc$ be finite linear categories and $H\colon \calaopp \Times \cala  
\To \calb$ be a functor whose (co)end exists. There are canonical isomorphisms
  \be
  F \big( \int_{a\in\cala}\! H(\ol a,a) \big) \,\cong \int_{a\in\cala} F(H(\ol a,a))
  \qquad{\rm for~any} ~~ F \iN \Funle(\calb,\calc) 
  \label{eq:end-Fleft}
  \ee
and
  \be
  G \big( \int^{a\in\cala}\!\! H(\ol a,a) \big) \,\cong \int^{a\in\cala}\! G(H(\ol a,a)) 
  \qquad{\rm for~any} ~~ G \iN \Funre(\calb,\calc) \,.
  \label{eq:coend-Gright}
  \ee
\end{Lemma}

As another simple relationship we mention that for any functor $H\colon \calaopp \Times \cala
\To \calb$ between finite linear categories there is a canonical isomorphism
  \be
  \ol{\int^{a\in\cala}\!\! H(\ol a,a)} \cong \int_{a \in \cala} H^{\opp}_{}(a,\ol a)
  \label{eq:opp-coend=end}
  \ee
with $H^\opp_{}\colon \cala \Times \calaopp \To \calbopp$ and the obvious relation between
the dinatural families of the coend and end.


\subsection{Generalized Yoneda lemmas}\label{sec:delt}

For any functor $F\colon \cala\To\vect$ we have
  \be
  \int_{b\in\cala} \HO\vect {\Ho\cala ab} {F(b)} 
  \stackrel{\eqref{eq:end-nat-trans}}\cong \Nat( \Ho\cala a-,F) \,\cong F(a) 
  \label{eq:Yon}
  \ee 
by the Yoneda lemma. This can be used to show that the Hom functor and the dual Hom functor
(as introduced in Definition \ref{Def:dualHom}) constitute units for a kind of convolution 
product of linear functors. This property
of the Hom and dual Hom functors plays a key role in various calculations below
(as well as e.g.\ in topological \cite{lyub11} and conformal \cite{fuSc22} quantum field
theory). Therefore we give an explicit proof, even though the result is known (see 
\cite[Cor.\,1.4.5\,\&\,Ex.\,1.4.6]{RIeh}):

\begin{Proposition}
Let $\cala$ and $\calc$ be finite linear categories and $F\colon \cala \To \calc$ 
a linear functor.
\\[2pt]
{\rm (i)}\, There is a natural isomorphism 
  \be
  \int^{a \in \cala}\!\! \Ho\cala a- \otimes F(a) \,\cong F 
  \label{eq:delta-property}
  \ee
of linear functors. In particular, the coend on the left hand side exists as an object of 
the category $\Fun(\cala,\calc)$ of linear functors from $\cala$ to $\calc$.
More specifically, functorially in $x \iN \cala$ the object $F(x)$ 
has the properties of a coend $ \int^{a \in \cala}\!\! \Ho\cala ax \oti F(a) $.
\\[3pt]
{\rm (ii)}\, There is a natural isomorphism 
  \be
  \int_{a \in \cala} \Hox\cala -a \otimes F(a) \,\cong F 
  \label{eq:codelta-property}
  \ee
of linear functors.
\end{Proposition}

\begin{proof} ~\\[3pt]
(i)\, For any $x \iN \cala$ and $c \iN \calc$ we have
  \be
  \bearl
  \HO\calc {\int^{a\in\cala} \Ho\cala ax \oti F(a)} c
  \dsty \stackrel{\eqref{eq:coend-hom}}\cong \!
  \int_{a\in\cala} \HO\calc {\Ho\cala ax \oti F(a)} c
  \Nxl2\dsty \hspace*{4.1em}
  \stackrel{\eqref{Hom(av,a'v')}}\cong\! \int_{a\in\cala} \Hox\cala ax \otik \Ho\calc {F(a)}c
  \,\cong \int_{a\in\cala} \HO{\!\vect} {\Ho\cala ax} {\Ho\calc {F(a)}c}
  \Nxl3 \hspace*{4.1em}
  \stackrel{\eqref{eq:end-nat-trans}} \cong \Nat\big( \Ho\cala -x , \Ho\calc {F(-)}c \big)
  \stackrel{\eqref{eq:Yon}} \cong \Ho\calc {F(x)}c 
  \eear
  \ee
natural in $x \iN \cala$. Here in the last line we use the (contravariant) Yoneda lemma.
The claim now follows by invoking the Yoneda lemma once again.
The universal dinatural transformation consists of morphisms from $\Ho\cala ax \oti F(a)$ 
to $F(x)$, for $a \iN \cala$. Using the isomorphism $\Ho\calc {\Ho\cala ax \oti F(a)}{F(x)}
\,{\cong}\, \Ho\ko {\Ho\cala ax} {\Ho\calc {F(a)}{F(x)}}$, this morphism comes from the map 
$f \,{\mapsto} F(f)$ for all $f \iN \Ho\cala ax$.
\\[3pt]
(ii)\, Analogously as in the case of (i) the proof follows from
  \be
  \bearll
  \HO\calc c {\int_{a\in\cala} \Hox\cala xa \oti F(a)}
  \!\!&\dsty
  \stackrel{\eqref{eq:coend-hom2r}}\cong \!
  \int_{a\in\cala}\! \HO\calc c {\Hox\cala xa \oti F(a)}
  \Nxl2 &\dsty
  \stackrel{\eqref{Hom(av,a'v')}}\cong \!\! 
  \int_{a\in\cala}\! \HO{\!\vect} {\Ho\cala xa} {\Ho\calc c{F(a)}}
  \stackrel{\eqref{eq:Yon}} \cong \Ho\calc c{F(x)} \,.
  \eear
  \ee
This time the covariant Yoneda lemma has been used, while for completing the proof one 
implements the contravariant Yoneda lemma.
Finally, the universal dinatural transformation for the end comes 
from the isomorphism $\Ho\calc {F(x)} {\Hox\cala xa \oti F(a)} \,{\cong}\,
\Ho\ko {\Ho\cala ax} {\Ho\calc {F(a)}{F(x)}}$ for $a \iN \cala$;
as in (i), this is induced by the map  $f \,{\mapsto}\, F(f)$ for all $f \iN \Ho\cala ax$.
\end{proof}

In the case $\calc \eq \vect$ the property \eqref{eq:codelta-property} reduces to the
statement of the Yoneda lemma as given in \eqref{eq:Yon}. Actually, as the proof indicates, 
it is merely a repackaging of that statement and could therefore just be referred to as a
Yoneda lemma itself, and analogously \eqref{eq:delta-property} as a co-Yoneda lemma
\cite{RIeh}. For concreteness, interpreting loosely ends and coends as analogues of integrals,
we will instead in the sequel refer to these results as the 
\emph{convolution properties} of the Hom functor and of the dual Hom functor, respectively. 

The convolution property \eqref{eq:delta-property} of the Hom functor can also be seen 
as a consequence of the coend 
formulas for Kan extensions, applied to an extension along the identity functor, which gives
$F(c) \,{=}\, \mathrm{Lan}_{\id}F(c)  \,{=}\, \int^{a\in\cala}\! \Ho\cala ac \oti F(a)$.


\subsection{A Peter-Weyl expression for (co)regular bimodules}

For $A$ and $B$ \findim\ \ko-algebras and $G$ a \ko-linear functor from $ A\Mod$ to $B\Mod$,
the object $G(A)$ has a natural structure of a $B$-$A$-bimodule. To see this, 
note that the right multiplication $r_\alpha\colon A \To A$ by $\alpha \iN A$ on $A$
is a morphism of left $A$-mo\-du\-les. As a consequence,
$G(r_\alpha)\colon G(A) \To G(A)$ is a morphism of left $B$-mo\-du\-les, and one can take 
the ring homomorphism $\alpha \,{\mapsto}\, G(r_\alpha) \iN \End(G(A))$
as the definition of the right $A$-module structure on $G(A)$. By construction, this
commutes with the left $B$-module structure on $G(A)$.

This bimodule structure on $G(A)$ is used in the following statement in which (unlike 
in the Eilenberg-Watts results) no exactness property needs to be assumed.

\begin{Proposition}\label{prop:end(GotikWee}
Let $A$ and $B$ be \findim\ algebras and $G\colon A\Mod \To B\Mod$ a linear functor.
The end of the functor 
  \be
  \bearll
  \widetilde G : & A\Mod \Times A\Mod^\opp \xrightarrow{~~~}\, B\BimoD A  
  \Nxl1 & \hspace*{5.5em}
  m \Times \ol n \,\longmapsto\, G(m) \otik n^\wee 
  \eear
  \label{eq:widetildeG}
  \ee
is given as a $B$-$A$-bimodule by
  \be
  \int_{m\in A\Mod} G(m) \otik m^\wee = G(A)
  \label{eq:end-G}
  \ee
with the natural $B$-$A$-bimodule structure on $G(A)$. Its dinatural family is
  \be
  i^A_m := G(\hat\rho_m) :\quad  G(A) \to G(m) \otik m^\wee
  \qquad{\rm with}\quad~
  \hat\rho_m := (\rho_m \otik \id_{m^\wee}) \circ (\id_A \oti \coev_m)
  \label{eq:i^A_m}
  \ee
  for $(m\,,\,\rho_m{:}\, A {\otimes} m {\rightarrow} m) \iN A\Mod$, 
with $\coev_m \iN \Ho\ko \ko{m\otik m^\wee}$ the coevaluation in $\vect$ and with
$m \otik m^\wee$ considered as an $A$-module via the action of $A$ on the left tensorand. 
\\[2pt]
Similarly, the coend of the functor $\widetilde G$ is
  \be
  \int^{m\in A\Mod}\!\! G(m) \otik m^\wee = G(A^\wee_{})
  \label{eq:coend=G(A*)}
  \ee
with dinatural family 
  \be
  i^{A^\wee}_m := G(\widetilde\rho_m) \qquad{\rm with}\quad~
  \widetilde\rho_m :=
  \big( \id_{A^\wee_{\phantom|}} \otik [ \ev_m \cir (\rho_m \otik \id_\mwee) ] \big)
  \circ ( \coev_{A^\wee_{}} \otik \id_m \otik \id_{m^\wee_{}} ) \,,
  \label{eq:i^A*_m}
  \ee
where $\ev_m$ is the evaluation in $\vect$.
\end{Proposition}

\begin{proof}
We show the statement for the end; the proof for the coend is analogous.
\\[2pt]
(i)\,
By linearity of $G$ there is a natural isomorphism 
$G(m \otik m^\wee) \,{\cong}\, G(m) \otik m^\wee$ of left $B$-mo\-du\-les. Since 
$\hat\rho_m$ is a morphism of left $A$-modules, $G(\hat\rho_m)$ is a morphism of left 
$B$-modules. 
{}From the functoriality of $G$ and the module property of $m$ it follows that 
  \be
  i_{m}^{A} \circ G(r_\alpha) = (\id_{G(m)}\otik  r_\alpha^\wee) \circ i_{m}^{A} 
  \label{eq:5}
  \ee
for all $\alpha \iN A$, which tells us that $i_m^A$ is a morphism of right $A$-modules.
Thus in total we have
 \be
 i_m^A \in
  \Holr BA {G(A)} {\widetilde G(m,\ol m)} \,,
  \ee
i.e.\ $i_m^A$ is well defined. 
Further, invoking the fact that $\rho_m$ is a bimodule morphism and the naturality of the
coevaluation one now sees that the family $i^A \eq (i^A_m)_{m\in A\Mod}$ is dinatural.
\\[2pt]
(ii)\,
Assume that $j_m^z \iN \Holr BA z {\widetilde G(m,\ol m)}$ is any dinatural
family from an object $z \iN B\BimoD A$ to $\widetilde G$.
Applying dinaturalness of $j^z$ to the morphisms $(\id_{G(A)} \otik \mu^\wee) \cir i_m^A$ 
for $m \iN A\Mod$ and $\mu^\wee \iN m^\wee$ gives
  \be
  (G(r_\mu) \otik \id_\Awee) \circ j^z_A = (\id_{G(m)} \otik r_\mu^\wee) \circ j^z_m \,.
  \label{eq:..jzA..jzm}
  \ee
We denote, for $w\iN\vect$, by $\ev_w\iN \Ho\ko {w^\wee \otik w}\ko$ the evaluation map and
by $f^\wee \eq (\ev_w \otik \id_{v^\wee}) \,{\circ} 
          $\linebreak[0]$
(\id_{w^\wee} \otik f \otik \id_{v^\wee})
\cir (\id_{w^\wee} \otik \coev_v)$ the $\vect$-dual of a linear map $f \iN \Ho\ko vw$.
Post-composing the equality \eqref{eq:..jzA..jzm} with $\id_{G(m)} \otik \eta_A^\wee$, 
where $\eta_A^\wee \eq \ev_A \cir (\id_\Awee \oti \eta_A)$ is the dual of the unit
$\eta_A \iN \Ho\ko \ko A$ of $A$, results in
 $  (G(r_\mu) \otik \eta_A^\wee) \cir j^z_A \eq (\id_{G(m)} \otik \mu^\wee) \cir j^z_m $
for all $\mu \iN m$, implying that
  \be
  (i_m^A \otik \eta_A^\wee) \circ j^z_A
  = (\id_{G(m)} \otik \id_{m^\wee_{}}) \circ j^z_m = j^z_m \,.
  \ee
Thus the dinatural family for $z$ is related to the one for $A$ by 
  \be
  j_m^z = i_m^A \circ \kappa_z
  \label{eq:jmz=imA.kappa}
  \ee
for all $m \iN A\Mod$, with $ \kappa_z \,{:=}\, (\id_{G(A)} \otik \eta_A^\wee) \cir j^z_A $.
\\[2pt]
(iii)\,
It remains to show that $\kappa_z$ is the unique morphism satisfying
\eqref{eq:jmz=imA.kappa} for all $m$.
This is seen by noting that $i_A^A \cir \kappa \eq j_A^z$ as a special case of 
\eqref{eq:jmz=imA.kappa}, and that, by the calculation
  \be
  \bearll
  (\id_{G(A)} \otik \eta_A^\wee) \circ i_A^A \!\!&
  = (\id_{G(A)} \otik \eta_A^\wee) \circ (\psi_A \otik \id_{A^\wee_{}})
  \circ (\id_{G(A)} \otik \coev_A)
  \Nxl3 &
  = \psi_A \circ (\id_{G(A)} \otik \eta_A^{})
  = G(r_{\eta_A^{}}^{}) = G(\id_A^{}) = \id_{G(A)}^{} \,,
  \eear
  \ee
$i_A^A$ has a left inverse.
\end{proof}

Taking $G$ to be the identity functor, as a special case of Proposition \ref{prop:end(GotikWee}
we have the following Peter-Weyl type result which provides a Morita invariant characterization 
of regular and co-regular bimodules:

\begin{cor}\label{cor:intAmod...=A*}
For any \findim\ \ko-algebra $A$ there are isomorphisms
  \be
  \int_{m\in A\Mod} m \otik \mwee \,\cong\, {}_A A_A
  \label{eq:end=A}
  \ee
and
  \be
  \int^{m\in A\Mod}\! m \otik \mwee \,\cong\, {}_A^{} A_A^\wee
  \label{eq:coend=A*}
  \ee
of $A$-bimodules. The dinatural transformations of the end and coend are given by the families
$ i_m^A \,{:=}\, (\rho_m \otik \id_{m^\wee}) \circ (\id_A \oti \coev_m) $ and
  \be
  i_m^{A^\wee} := (\id_\Awee \oti \ev_m) \circ (\id_\Awee \oti \rho_m \oti \id_\mwee)
  \circ (\coev_A \oti \id_m \oti \id_\mwee)
  \label{eq:i_m}
  \ee
of linear maps for $(m,\rho_m) \iN A\Mod$, respectively.
\end{cor}

As already pointed out, no exactness of the functors is assumed in these results. Since for 
any finite linear category $\calm$ we can choose an algebra $A$ and an equivalence 
$\calm \,{\simeq}\, A\Mod$, we 
can transport them to the setting of finite linear categories. For doing so we also need
an analogue of the tensor product over $\ko$; this is provided by the Deligne product.

Recall \cite[Sect.\,5]{deli} that the Deligne product $\cala \boti \calb$ of two finite linear
categories is the linear category that is universal for right exact functors with domain 
$\cala \Times \calb$. That is, there is a functor 
$\boxtimes\colon \cala \Times \calb \To \cala \boti \calb$, exact in each variable,
such that for any linear category $\calx$ every right exact (in both variables) functor
$F$ from $\cala \Times \calb$ factorizes uniquely (up to unique natural isomorphism)
as $F \eq \tilde F \cir {\boxtimes}$ with 
$\tilde F\colon \cala \boti \calb \To \calx$ being right exact. Thus in particular we have 
an equivalence $\Funre(\cala\Times\calb, \calx) \,{\simeq}\, \Funre(\cala\boti\calb, \calx)$,
functorial in $\calx$,
of categories of right exact functors. By considering opposite functors and using that
the opposite of a right exact functor is left exact, the Deligne product is universal with 
respect to \emph{left} exact functors as well.
Also note that for $A$ and $B$ \findim\ \ko-algebras we have linear isomorphisms
  \be 
  A\Mod \boxtimes \moD B
  \,\cong\, A\Mod \boxtimes \Bop\Mod
  \,\cong\, (A \otik \Bop)\Mod
  \,\cong\, A\BimoD B \,.
  \ee

The Peter-Weyl results above thus imply

\begin{cor}\label{cor:coend-exists}
For every linear functor $F\colon \calm \To \caln$ between finite linear categories, 
the end $\int_{m \in \calm} F(m) \boti \overline{m}$ 
and the coend $\int^{m \in \calm}\! F(m) \boti \overline{m}$ in $\caln \boti \calmopp$ exist.
\end{cor}

\begin{Remark}
For the case that $A$ is a Hopf algebra, the assertion of Corollary \ref{cor:intAmod...=A*} 
has already been shown in \cite[Prop.\,A.3]{fuSs3}. The proof in \cite{fuSs3} is obtained 
from the one above by considering the right $A$-action on $A$ that is obtained from the
left action by means of the antipode of $A$.
\end{Remark}


\section{Eilenberg-Watts calculus for finite linear categories}

If $\cala$ and $\calb$ are finite linear categories, then there are \findim\ algebras
$A$ and $B$ such that $\cala \,{\simeq}\, A\Mod$ and $\calb \,{\simeq}\, B\Mod$.
The Eilenberg-Watts results of Lemma \ref{lem:shimi7:2.6} then imply that the functor 
category $\Funre(\cala,\calb)$ is equivalent to $B$-$A\Bimod$, which in turn is equivalent 
to the Deligne product 
$\calaopp \boti \calb $. In principle the resulting equivalence
  \be
  \Funre(\cala,\calb) \,\simeq\, \calaopp \boti \calb
  \ee
of finite linear categories might depend on the choice of algebras $A$ and $B$.
In the sequel we will analyze this equivalence, as well as a corresponding equivalence for 
left exact functors, without having to refer to any choice of algebras.
Thus in particular these equivalences do not depend on those individual algebras, but only on 
their Morita classes.
Our discussion will strongly rely on the use of ends and coends, the relevance of which 
may be traced back to Proposition \ref{prop:end(GotikWee}.

\subsection{A triangle of adjoint equivalences}\label{subsec:adjeqv}

In this subsection we establish, for any two finite \ko-linear categories $\cala$ and 
$\calb$, adjoint equivalences between the Deligne product $\calaopp \boti \calb$
and the categories $\Funle(\cala,\calb)$ and $\Funre(\cala,\calb)$ of left exact
and right exact functors, respectively, from $\cala$ to $\calb$.
These equivalences imply in particular that these functor categories are finite linear 
categories. For the case of $\Funle(\cala,\calb)$ this equivalence has been established in 
\cite[Rem.\,2,2(i)]{etno2} and \cite[Lemma\,3.2]{shimi7}; the result for $\Funre(\cala,\calb)$ 
follows by considering in addition to the Hom functor also the dual Hom functor.

We start by 

\begin{Definition}
Let $\cala$ and $\calb$ be finite linear categories. The (abstract) \emph{Eilenberg-Watts
functors} for $\cala$ and $\calb$ are the following four linear functors:
  \be
  \begin{array}{lrcl}
  \Phile_{} \equiv \Phile_{\!\cala,\calb}:
  & \calaopp \boti \calb & \longrightarrow & \Funle(\cala,\calb) 
  \Nxl1
  & \ol a \boti b & \longmapsto & \Ho\cala a- \oti b \,,
  \Nxl4
  \Psile_{} \equiv \Psile_{\!\cala,\calb}:
  & \Funle(\cala,\calb) & \longrightarrow & \calaopp \boti \calb
  \Nxl1
  & F & \longmapsto & \int^{a\in\cala} \ol a \boti F(a) \,,
  \Nxl4
  \Phire_{} \equiv \Phire_{\!\cala,\calb}:
  & \calaopp \boti \calb & \longrightarrow & \Funre(\cala,\calb) 
  \Nxl1
  & \ol a \boti b & \longmapsto & \Hox\cala -a \oti b \,,
  \Nxl4
  \Psire_{} \equiv \Psire_{\!\cala,\calb}:
  & \Funre(\cala,\calb) & \longrightarrow & \calaopp \boti \calb
  \Nxl1
  & G & \longmapsto & \int_{a\in\cala} \ol a \boti G(a) \,.
  \eear
  \label{Phile...Psire}
  \ee
\end{Definition}

Note that the functors $\Phile$ and $\Phire$ take a simple form only when applied 
to $\boxtimes$-factorized objects. We will have to apply them also on objects that are ends
or coends, in which case we write for example
  \be
  \Phile(\int^a \ol a \boti F(a))
  = \Big[ \Hom_\cala(\reflectbox?,-) \oti ? \Big] (\int^a \ol a \boti F(a)) \,.
  \ee
We also abbreviate
  \be
  \bearll &
  \Gammarl := \Phire \circ \Psile ~:~ \Funle(\cala,\calb) \to \Funre(\cala,\calb) 
  \Nxl4 {\rm and}\quad &
  \Gammalr := \Phile \circ \Psire ~:~ \Funre(\cala,\calb) \to \Funle(\cala,\calb) \,.
  \eear
  \label{eq:Gammalr,Gammarl}
  \ee
Using that the Hom functor, being left exact, preserves ends and that the dual Hom functor
preserves coends, this implies
  \be
  \Gammarl(F) = \int^{a\in\cala}\!\! \Hox\cala -a \oti F(a) \qquad{\rm and}\qquad
  \Gammalr(G) = \int_{a\in\cala} \Ho\cala a- \oti G(a)
  \label{eq:Gammarl,Gammalr}
  \ee
for $F \iN \Funle(\cala,\calb)$ and for $G \iN \Funre(\cala,\calb)$, respectively.
In the definition of $\Phile$ we use that the Deligne product is exact and universal with 
respect to left exact functors. In the definition of $\Psile$ and $\Psire$ we use the fact,
seen in Corollary \ref{cor:coend-exists}, that the coend and end, respectively, exist,
and that they are functorial with respect to natural transformations.

For ease of reference we collect these functors in the following diagram:
 \\[-1.05em]
  \be
  \begin{tikzcd}[row sep=11ex]
  ~ & \calaopp \boti \calb ~ \ar{dl}[xshift=-2pt]{\Phile} \ar[xshift=-2pt]{dr}[swap]{\Phire}
  & ~ \\
  \Funle(\cala,\calb) \ar[yshift=3pt]{rr}[yshift=1.3pt]{\Gammarl} \ar[xshift=-12pt]{ur}[xshift=2pt]{\Psile}
  & ~ & \Funre(\cala,\calb) \ar[yshift=-3pt]{ll}{\Gammalr} \ar[xshift=12pt]{ul}[swap]{\Psire}
  \end{tikzcd}
  \label{tria:AoppB-lex-rex}
  \ee

\begin{thm}\label{thm:triangle}[Categorical Eilenberg-Watts theorem] \\
For any pair of finite \ko-linear categories $\cala$ and $\calb$ the functors in the 
triangle \eqref{tria:AoppB-lex-rex} constitute quasi-inverse pairs of adjoint equivalences
  \be
  \Funle(\cala,\calb) \,\simeq\, \calaopp \boti \calb \,\simeq\, \Funre(\cala,\calb) \,.
  \label{Funle=boti=Funre}
  \ee
\end{thm}

\begin{proof}
(i)\, That the functor $\Phile$ is an equivalence is shown in \cite[Lemma\,3.2]{shimi7},
and that $\Psile$ is a quasi-inverse of $\Phile$ in \cite[Lemma.\,3.3]{shimi7}. 
\\[3pt]
(ii)\, To obtain the second equivalence in \eqref{Funle=boti=Funre} we apply the first one 
after taking opposites. Note that for $\ol a \iN \cala^\opp$ and $v \iN \vect$ we have
  \be
  \ol a \oti v = \ol {a \oti v^\wee} ~\iN \calaopp \,,
  \label{eq:abar.v=a.v*bar}
  \ee
and that by Lemma \ref{lem:lexrexbimod}
$ \Funre(\cala,\calb)$ is equivalent to $\big( \Funle(\calaopp,\calbopp) \big)^\opp_{}.$  
The equivalences $\Phile$ and $\Psile$ give a pair of quasi-inverse equivalences
  \be
  \bearl
  (\Phile)^\opp_{}:\quad
  \calaopp \boti \calb \,\simeq\, \big( \cala \boti \calbopp \big)^\opp_{}
  \xrightarrow{~\simeq~}\, \big( \Funle(\calaopp,\calbopp) \big)^\opp_{} \qquad{\rm and}
  \Nxl2
  (\Psile)^\opp_{}:\quad \big( \Funle(\cala^\opp,\calb^\opp) \big)^\opp_{}
  \xrightarrow{~\simeq~}\, \big( \cala \boti \calb^\opp \big)^\opp_{}
  \,\simeq\, \calaopp \boti \calb \,.
  \eear
  \label{eq:Psile-opp}
  \ee
The so defined functor $(\Phile)^\opp_{}$  maps the object $\ol a \boti b \iN \cala^\opp
\boti \calb$ to the opposite of the functor $\Hom_{\cala^\opp}(a,-) \oti \ol b$, i.e.\
$ (\Phile)^\opp_{}(\ol a \boti b) \,{\stackrel{\eqref{eq:abar.v=a.v*bar}}=}\, \big( \Ho\cala 
- a \big)^\wee \oti b \eq \Phire(\ol a \boti b) $, while the functor $(\Psile)^\opp_{}$ 
equals $\Psire$: it maps $\ol F$ with $F \iN \Funle(\calaopp,\calbopp)$ to
  \be
  (\Psile)^\opp_{}(\ol F) = \ol{ \int^{\ol a \in \calaopp}\!\! a \boti \ol F(\ol a) }
  \,\cong \int_{a \in\cala} \ol a \boti F(a) = \Psire(\ol F) \,.
  \ee
(iii)\, Using formula \eqref{eq:end-nat-trans} we have
$ \int_{a'\in\cala} \Ho{\calaopp_{}\boxtimes\calb} {\ol{a'}\boti F(a')} {\ol a \boti b}
\,{\cong}\, \Nat(F,\Phile(\ol a {\otimes} b))$ and hence
  \be
  \hspace*{-1em}\bearll
  \Ho{\Fun(\cala,\calb)} F {\Phile(x)} \equiv \Nat(F,\Phile(x)) \!\!&\dsty
  \cong \int_{a\in\cala} \Hox{\calaopp\boxtimes\calb} {\ol a\boti F(a)} x
  \Nxl2&
  \cong \Ho{\!\calaopp\boxtimes\calb} {\int^{a\in\cala} \ol a\boti F(a)} x
  \,= \Ho{\calaopp\boxtimes\calb} {\Psile(F)} x
  \eear
  \label{eq:proof-adj-l}
  \ee
for $x \iN \calaopp\boti\calb$.
Thus $\Psile$ is left adjoint to $\Phile$. Similarly, the calculation
  \be
  \bearll
  \Ho{\Fun(\cala,\calb)} {\Phire(\ol a\boti b)} G
  \!\!& \dsty
  \cong \int_{a'\in\cala} \HO\calb {\Hox\cala {a'}a \oti b} {G(a')}
  \cong \int_{a'\in\cala} \Ho\cala {a'}a \otik \Ho\calb b{G(a')}
  \Nxl2&\dsty
  \cong \int_{a'\in\cala} \Ho{\calaopp\boxtimes\calb} {\ol a\boti b} {\ol{a'}\boti G(a')}
  \Nxl2&
  \cong \Ho{\calaopp\boxtimes\calb} {\ol a\oti b} {\int_{a'\in\cala}\ol{a'}\boti G(a')}
  \,= \Ho{\calaopp\boxtimes\calb} {\ol a\boti b} {\Psire(G)}
  \eear
  \label{eq:proof-adj-r}
  \ee
for $a\iN\cala$ and $b\iN\calb$
shows that (also invoking universality of the Deligne product) $\Psire$ is right adjoint 
to $\Phire$. Since the respective pairs of functors are equivalences, the adjunctions
are two-sided.
\\[3pt]
(iv)\,
That the functors $\Gammarl$ and $\Gammalr$ form an adjoint equivalence follows directly
by composing the relevant statements for the functors $\Phile,\,\Psile$ and 
$\Phire,\,\Psire$.
\end{proof}

The following rewritings will be convenient:

\begin{Lemma}\label{lem:Psile-as-end}
For any functor $F \iN \Funle(\cala,\calb)$ we have
  \be
  \Psile(F) \cong \int_{a\in\cala} \ol a \boti \Gammarl(F)(a) \,,
  \label{eq:Psile-as-end}
  \ee
and for any functor $G \iN \Funre(\cala,\calb)$ we have
  \be
  \Psire(G) \cong \int^{a\in\cala}\! \ol a \boti \Gammalr(G)(a) \,.
  \label{eq:Psire-as-coend}
  \ee
\end{Lemma}

\begin{proof}
The right hand side of \eqref{eq:Psile-as-end} is nothing but
$ \Psire(\Gammarl(F)) \,{\cong}\, \Psire \cir \Phire \cir \Psile(F) $ and thus,
$\Psire$ and $\Phire$ being quasi-inverse,  equals the left hand side.  
The isomorphism \eqref{eq:Psire-as-coend} follows analogously by recognizing the
right hand side as $\Psile(\Gammalr(G))$.
\end{proof}


\subsection{Hom pairings, compositions and (co)ends for finite categories}

In the proof of Theorem \ref{thm:triangle} we encountered the relations 
\eqref{eq:proof-adj-l} and \eqref{eq:proof-adj-r} for morphism spaces which involve
objects of the form $\Psile(F)$ or $\Psire(G)$. 
As we deal with an adjoint equivalence, $\Psile$ is not only left, but also right adjoint 
to $\Phile$. As a consequence there also exist relations for morphism spaces in which with 
respect to those in \eqref{eq:proof-adj-l} and \eqref{eq:proof-adj-r} the domain and codomain 
are interchanged. This provides the following `opposed' variant of the basic (co)end
preservation statement of Lemma \ref{lem:basic-coend-pres}, which is valid in the finite
linear case considered here; this result will be a rich source of identities relevant to us.

\begin{Proposition}\label{prop:matrixelements}
Let $\cala$ and $\calc$ be finite linear categories. Then for $F \iN \Funle(\cala,\calc)$
we have
  \be
  \HO{\!\cala^\opp\boxtimes\calc} {\ol a \boti c} {\int^{b\in\cala} \ol b \boti F(b)}
  \,\cong \int^{b\in\cala}\!\! \Ho{\cala^\opp\boxtimes\calc} {\ol a \boti c} {\ol b \boti F(b)} \,,
  \label{homAoppC-coend}
  \ee
while for $G \iN \Funre(\cala,\calc)$ we have
  \be
  \HO{\Calaopp\boxtimes\calc} {\int_{b\in\cala} \ol b \boti G(b)} {\ol a \boti c}
  \,\cong \int^{b\in\cala}\!\! \Ho{\calaopp\boxtimes\calc} {\ol b \boti G(b)} {\ol a \boti c} \,,
  \label{homAoppC-end}
  \ee
respectively, natural in $a\iN\cala$ and $c\iN\calc$.
\end{Proposition}

\begin{proof}
(i)\, To obtain the isomorphism \eqref{homAoppC-coend} we use that $\Psile$ has $\Phile$ as 
a quasi-inverse and invoke \eqref{eq:end-nat-trans} and the (Yoneda) \delt\ property of the
dual Hom functor to compute \be
  \bearl
  \Ho{\calaopp\boxtimes\calc} {\ol a \boti c} {\Psile(F)}
  \,\cong\, \Ho{\!\Funle(\cala,\calc)} {\Phile(\ol a \boti c)} F
  \,=\, \Nat\big( \Ho\cala a- \oti c , F \big)
  \Nxl1 \dsty \hspace*{3.5em}
  \cong\, \int_{b\in\cala} \HO\calc {\Ho\cala ab \oti c} {F(b)}
  \,\cong\, \int_{b\in\cala} \Hox\cala ab \otik \Ho\calc c{F(b)}
  \,\cong\,  \Ho\calc c{F(a)} \,.
  \eear
  \ee
On the other hand, the (co-Yoneda) \delt\ property of the Hom functor implies that 
the same result can be obtained as follows:
  \be
  \int^{b\in\cala} \Ho{\calaopp\boxtimes\calc} {\ol a \boti c} {\ol b \boti F(b)}
  \,\cong \int^{b\in\cala}\! \Ho\cala ba \otik \Ho{\!\cala} {F^\la(c)}b
  \,\cong\, \Ho\calc c{F(a)} \,.
  \label{eq:last-iso}
  \ee
The claim \eqref{homAoppC-coend} now follows by inserting the definition of $\Psile$.
\\[3pt]
(ii)\,
Analogously we calculate
  \be
  \bearll
  \Ho{\!\cala^\opp_{}\boxtimes\calc} {\Psire(G)} {\ol a \boti c}
  \!\! &
  \cong\, \Ho{\Funre(\cala,\calc)} G {\Phire(\ol a \boti c)}
  \,=\, \Nat\big( G , \Hox\cala -a \oti c \big)
  \Nxl2 &\dsty
  \cong\, \int_{b\in\cala} \HO\calc {G(b)} {\Hox\cala ba \oti c}
  \,\cong\, \int_{b\in\cala} \Ho\calc {G(b)}c \otik \Hox{\!\cala} ba
  \Nxl2 &\dsty
  \cong\, \int_{\ol b\in\calaopp} \Ho\calc {G(b)}c \otik \Hox\calaopp {\ol a}{\ol b}
  \,\cong\, \Ho\calc {G(a)}c 
  \eear
  \ee
as well as (using now in addition that, being right exact, by Corollary 
\ref{cor:leftexact-etc} $G$ has a right adjoint)
  \be
  \bearll\dsty
  \int^{b\in\cala}\!\! \Ho{\calaopp\boxtimes\calc} {\ol b \boti G(b)} {\ol a \boti c}
  \!\!&\dsty
  \cong \int^{b\in\cala}\!\! \Ho\cala ab \otik \Ho{\!\cala} b{G^\ra(c)}
  \Nxl2 &
  \cong\, \Ho\cala a{G^\ra(c)} \,\cong\, \Ho\calc {G(a)}c \,.
  \eear
  \ee
Again the claim now follows by comparison, after inserting the definition of $\Psire$.
\end{proof}

The left hand side of each of the isomorphisms \eqref{homAoppC-coend} and \eqref{homAoppC-end}
defines a left exact (contravariant, respectively covariant) functor from $\calaopp\Times\calc$
to \vect, and thus, by the universal property of the Deligne product, a corresponding functor
from $\calaopp\boti\calc$ to \vect. 
It may, however, happen that the so defined functor cannot be expressed as a coend in the
same way as on the right hand side of \eqref{homAoppC-coend} and \eqref{homAoppC-end}. But
if we take the coend on the right hand side in the category of left exact functors, then we
do get such an expression. Using, as in \cite{lyub11}, the symbol $\oint$ for such coends 
in categories of left exact functors, we have

\begin{cor}\label{cor:oint}
Let $\cala$ and $\calc$ be finite linear categories. Then for any $F \iN \Funle(\cala,\calc)$
there is an isomorphism 
  \be
  \HO{\!\cala^\opp\boxtimes\calc} - {\int^{b\in\cala} \ol b \boti F(b)}
  \,\cong \oint^{b\in\cala}\!\! \Ho{\cala^\opp\boxtimes\calc} - {\ol b \boti F(b)} \,,
  \label{eq:oint1}
  \ee
in $\Funle(\cala\boti\calcopp,\vect)$, 
while for any $G \iN \Funre(\cala,\calc)$ there is an isomorphism
  \be
  \HO{\Calaopp\boxtimes\calc} {\int_{b\in\cala} \ol b \boti G(b)} -
  \,\cong \oint^{b\in\cala}\!\! \Ho{\calaopp\boxtimes\calc} {\ol b \boti G(b)} -
  \label{eq:oint2}
  \ee
in $\Funle(\calaopp\boti\calc,\vect)$.
\end{cor}

\begin{proof}
The coends on the right hand side of these expressions exist as objects of the relevant
functor categories. Moreover, when restricted to $\boxtimes$-factorized objects, by
Proposition \ref{prop:matrixelements} they coincide with the functors on the left hand side.
The statement thus follows by the universal property (for left exact functors)
of the Deligne product.
\end{proof}

We will frequently use Proposition \ref{prop:matrixelements} in the form that is more
directly seen in the proof, i.e. as

\begin{cor}
For $\cala$ and $\calc$ be finite linear categories, we have
  \be
  \Ho{\calaopp\boxtimes\calc} {\ol a \boti c} {\Psile(F)} \,\cong\, \Ho\calc c{F(a)}
  \quad{\rm for} ~~ F \iN \Funle(\cala,\calc)
  \label{ac,PsilF=c,Fa}
  \ee
and
  \be
  \Ho{\calaopp\boxtimes\calc} {\Psire(G)} {\ol a \boti c} \,\cong\, \Ho\calc {G(a)}c
  \quad{\rm for} ~~ G \iN \Funre(\cala,\calc) \,.
  \label{PsirG,ac=Ga,c}
  \ee
\end{cor}

Invoking the Yoneda lemma, this determines $\Psile(F)$ and $\Psire(G)$. The same 
type of arguments gives the following result for the composition of functors,
by which coends and ends can be understood as a kind of categorified matrix multiplication:
	
\begin{cor}\label{cor:int=G'G-etc}
Let $F\colon \cala\To\calb$ and $F'\colon \calb\To\calc$ be left exact functors,
and $G\colon \cala\To\calb$ and $G'\colon \calb\To\calc$ right exact functors.
Then there are isomorphisms
  \be
  \int^{b\in\calb}\!\! \Ho{\cala^\opp\boxtimes\calb} {\ol a\boti b} {\Psile_{\!\cala,\calb}(F)}
  \otik \Ho{\calb^\opp\boxtimes\calc} {\ol b\boti c} {\Psile_{\calb,\calc}(F')}
  \cong \Ho\calc c {F'\cir F(a)}
  \label{int=F'F}
  \ee
and
  \be
  \int^{b\in\calb}\! \Ho{\cala^\opp\boxtimes\calb} {\Psire_{\!\cala,\calb}(G)} {\ol a \boti b}
  \otik \Ho{\calb^\opp\boxtimes\calc} {\Psire_{\calb,\calc}(G')} {\ol b \boti c}
  \cong \Ho\calc {G'\cir G(a)} c 
  \label{int=G'G}
  \ee
natural in $a\iN\cala$ and in $c\iN\calc$.
Analogously, there are natural isomorphisms
  \be
  \int_{b\in\calb} \Hox{\cala^\opp\boxtimes\calb} {\ol a\boti b} {\Psile_{\!\cala,\calb}(F)}
  \otik \Hox{\calb^\opp\boxtimes\calc} {\ol b\boti c} {\Psile_{\calb,\calc}(F')}
  \cong \Hox\calc c {F'\cir F(a)}
  \label{intdual=F'F}
  \ee
and
  \be
  \int_{b\in\calb} \Hox{\cala^\opp\boxtimes\calb} {\Psire_{\!\cala,\calb}(G)} {\ol a \boti b}
  \otik \Hox{\calb^\opp\boxtimes\calc} {\Psire_{\calb,\calc}(G')} {\ol b \boti c}
  \cong \Hox\calc {G'\cir G(a)} c \,.
  \label{intdual=G'G}
  \ee
~
\end{cor}

\begin{proof}
Invoking the isomorphism \eqref{ac,PsilF=c,Fa},    
the coend on the left hand side of
\eqref{int=F'F} becomes $ \int^{b\in\calb} \Ho\calb  b {F(a)} \otik \Ho\calc c {F'(b)} $.
After using that the functor $F'$ has a left adjoint $F'^{\,\la}$ we can apply the \delt\
property of the Hom functor to reduce this expression to
$ \Ho\calb {F'^{\,\la}(c)} {F(a)} \,{\cong}\, \Ho\calc c {F' \cir F(a)} $.
This proves \eqref{int=F'F}. To show the isomorphism \eqref{int=G'G} we use analogously 
the result \eqref{homAoppC-end} together with the fact that $G'$ has a right adjoint.
The other two isomorphisms are shown by the same type of argument.
\end{proof}


\subsection{Adjunctions and (co)ends}

Next we establish a compatibility between taking (co)ends and adjunctions.

\begin{Lemma}\label{lem:adj4(co)ends}
Let $\cala$ and $\calc$ be finite linear categories. For $F \iN \Funle(\cala,\calc)$
and $G \iN \Funre(\cala,\calc)$ there are isomorphisms
  \be
  \Psile(F) \equiv \int^{a\in\cala}\! \ol a \boti F(a)
  \,\cong\, \int^{c\in\calc}\! \ol{F^\la(c)} \boti c
  \label{eq:coend-vs-Fl}
  \ee
and
  \be
  \Psire(G) \equiv \int_{a\in\cala}\! \ol a \boti G(a)
  \,\cong\, \int_{c\in\calc}\! \ol{G^\ra(c)} \boti c 
  \label{eq:end-vs-Gr}
  \ee
of coends and of ends in $\calaopp \boti \calc$, respectively.  
\\
Moreover, for any two left exact functors $F_1\colon \cala \To \calb$ and 
$F_2\colon \calb \To \calc$ the two isomorphisms 
  \be
  \int^{a\in\cala}\! \ol a \boxtimes F_2 \cir F_1(a)
  \, \xrightarrow{~\simeq~} \int^{c\in\calc}\! \ol{F_1^\la {\circ}\,F_2^{\la}(c)} \boti c
  \ee
resulting from the isomorphism \eqref{eq:coend-vs-Fl} are equal.
Analogously also the isomorphism in \eqref{eq:end-vs-Gr} is coherent in this sense. 
\end{Lemma}

\begin{proof}
Applying the functor $\Phile$ to the left hand side of \eqref{eq:coend-vs-Fl} gives
$\Phile \cir \Psile(F) \,{\cong}\,F$. The same functor is obtained on the right hand side:
using that $\Phile$, being an equivalence, is right exact and thus preserves the coend, we have
  \be
  \bearll
  \Phile(\int^{c\in\calc} \ol{F^\la(c)} \boti c) \!\!&\dsty
  \cong \int^{c\in\calc}\! \Phile(\ol{F^\la(c)} \boti c)
  \,= \int^{c\in\calc}\!\!\! \Ho\cala{F^\la(c)}- \oti c
  \Nxl2&\dsty
  \cong \int^{c\in\calc}\!\! \Ho\calc c {F(-)} \oti c 
  \stackrel{\eqref{eq:delta-property}}\cong F \,.
  \eear
  \ee
The claim thus follows directly from the fact that $\Phile$ is an equivalence.
The isomorphism \eqref{eq:end-vs-Gr} is analogously obtained from the isomorphism
  \be
  \bearll
  \Phire(\int_{c\in\calc} \ol{G^\ra(c)} \boti c) \!\!&\dsty
  \cong \int_{c\in\calc}\! \Phire(\ol{G^\ra(c)} \boti c)
  \,= \int_{c\in\calc} \Hox\calc -{G^\ra(c)} \oti c
  \Nxl2&\dsty
  \cong \int_{c\in\calc} \Hox\calc {G(-)}c \oti c 
  \stackrel{\eqref{eq:codelta-property}}\cong G
  \eear
  \ee
of functors.
The coherence of these isomorphisms with respect to the composition of functors follows 
directly from the definition of the isomorphisms above. 
\end{proof}

For a right exact functor $G$, the functor $\Gammalr(G)$ is left exact and thus has a left 
adjoint, and analogously for a left exact functor $F$. 
Using adjoints and the triangle \eqref{tria:AoppB-lex-rex} of equivalences one obtains
two potentially different equivalences from $\Funre(\cala,\calb)$ to $\Funre(\calb,\cala)$,
namely mapping a right exact functor $G$ either to $\Gammarl(G^{\ra})$ or to
$(\Gammalr(G))^{\la}$. The next result, which uses Lemma \ref{lem:adj4(co)ends}, shows
that these two functors in fact coincide, and likewise do the analogous equivalences of
left exact functors.

\begin{cor}\label{cor:GammarlGra}
For any functor $G \iN \Funre(\cala,\calb)$ the functor $\Gammarl(G^\ra) \iN \Funre(\calb,\cala)$
is left adjoint to the functor $\Gammalr(G) \iN \Funle(\cala,\calb)$,
and for any $F \iN \Funle(\cala,\calb)$ the functor $\Gammalr(F^\la) \iN \Funle(\calb,\cala)$
is right adjoint to $\Gammarl(F) \iN \Funre(\cala,\calb)$.
\end{cor}

\begin{proof}
The statements follow by direct computation. We have
  \be
  \bearll
  \Ho\calb b {\Gammalr(G)(a)} \!\! &
  =\, \HO\calb b {\int_{c\in\cala}\, \Ho\cala ca \oti G(c)}
  \dsty
  \cong \int_{c\in\cala} \Ho\cala ca \otik \Ho\calb b {G(c)}
  \Nxl3&
  \cong\, \HO\Cala {\int^{c\in\cala} \Hox\calb b {G(c)} \oti c} a
  \Nxl3&
  \!\!\!\stackrel{\eqref{eq:coend-vs-Fl}}\cong\,
  \HO\Cala {\int^{d\in\calb} \Hox\calb bd \oti G^\ra(d)} a
  \,\cong\, \Ho\cala {\Gammarl(G^\ra)(b)} a
  \eear
  \label{eq:adj4Gamma()}
  \ee
for all $a \iN \cala$ and $b \iN \calb$;
this shows the first of the two claimed adjunctions. The second adjunction can be
seen analogously, or alternatively by re-interpreting \eqref{eq:adj4Gamma()}, with
$G^\ra \iN \Funre(\calb,\cala)$ playing the role of $F$.
\end{proof}


\subsection{Coends over Deligne products}

In the sequel we consider coends taken over a Deligne product of finite linear categories. 
Recall that for $\cala$ and $\calb$ finite linear categories, the functor categories 
$\Funle(\cala,\calb)$ and $\Funre(\cala,\calb)$ are finite linear as well. Also recall the 
notation $\oint^{\scriptscriptstyle\bullet}$ for coends in categories of left exact functors
which was introduced in Corollary \ref{cor:oint}; analogously we write
$\oint_{\!\scriptscriptstyle\bullet}$ for an end taken in a category of right exact functors.
Our result is based on the following statement on the coend of the Hom functor of a 
Deligne product.

\begin{Lemma}\label{lem:Deligne1}
For $\calc \eq \cala \boti \calb$ a Deligne product of finite linear categories there are 
isomorphisms 
  \be
  \int^{c \in \calc}\! \Ho{\calc^\opp\boti\calc} {\ol y \boti x} {\ol c \boti c}
  \,\cong\, \Ho\calc xy 
  \,\cong \oint^{a \in \cala} \!\!\! \int^{b \in \calb}\! \Ho{\calc^\opp\boti\calc} 
  {\ol y \boti x} {\ol{a \boti b} \,\boti\, a \boti b}
  \label{eq:2}
  \ee
and
  \be
  \int_{c \in \calc}\, \Hox{\calc^\opp\boti\calc} {\ol y \boti x} {\ol c \boti c}
  \,\cong\, \Hox\calc xy \,\cong 
  \oint_{a \in \cala} \! \int_{b \in \calb}\, \Hox{\calcopp\boxtimes\calc} 
  {\ol y \boti x} {\ol{a \boti b} \,\boti\, a \boti b}
  \label{eq:2e}
  \ee
natural in $x \iN \calc$ and  $\ol y \iN \calcopp$.
\end{Lemma}

\begin{proof}
The first of the two isomorphisms \eqref{eq:2} follows from the \delt\ property 
\eqref{eq:delta-property} of the Hom functor after noting that $\Ho{\calcopp\boxtimes\calc}
{\ol x \boti y} {\ol c \boti c} \,{\cong}\, \Ho{\calc}{y}{c} \otik \Ho{\calc}{c}{x}$. 
To establish the second isomorphism, we choose \findim\ algebras $A$ and $B$ with equivalences 
$\cala \,{\simeq}\, A\Mod$ and $\calb \,{\simeq}\, \moD B$. Then we also have an equivalence
$\cala \boti \calb \,{\simeq}\, A \BimoD B$, and the right hand side of formula \eqref{eq:2}
gives the expression $\oint^{m \in A \Mod}\!\! \int^{n \in \moD B} \Hom_{A|B}(x,m \otik n) \otik 
\Hom_{A|B}(m \otik n, y)$ for $x\eq {}_Ax_B,$ $y \eq {}_Ay_B \iN A \BimoD B$, 
$m \eq {}_Am \iN A\Mod$ and $n \eq n_B \iN \moD B$. The existence of these coends as well 
as the equivalence to the Hom functor is shown by the following chain of isomorphisms:
  \be
  \hspace*{-.5em}
  \bearll \multicolumn{2}{l} {\dsty
  \oint^{m \in A\Mod}\!\!\! \int^{n \in\moD B}\!\!\!
  \Ho{A|B} x {m \otik n} \otik \Ho{A|B} {m \otik n} y }
  \Nxl2 \quad~ & \dsty
  \cong \oint^{m\in A\Mod}\!\!\! \int^{n\in\moD B}\!\!\!
  \Ho B {m^* \otA x} n \otik \Ho B n {(y^* \otA m)^{*}}
  \Nxl2 &\dsty
  \cong \oint^{m \in A\Mod}\!\!\! \Ho B {m^* \otA x} {(y^* \otA m)^*} 
  \,\cong \oint^{m \in A\Mod}\!\!\! \Ho{A|B} {(m \otik m^*) \otA x} y 
  \Nxl1 &\dsty 
  \cong \oint^{m \in A\Mod}\!\!\! \Ho{A|A} {m \otik m^* \otA (x \otB  y^*)} A 
  \Nxl2 &
  \!\!\! \stackrel{\eqref{eq:oint2}} \cong\!\!\!
  \Ho{A|A} {\int_{m \in A\Mod} m \otik (m^* \otA (x \otB  y^*))} A 
  \,\cong \Ho{A|A} {A \otA x \otB y^*} A \,\cong \Ho{A|B} x y \,.
  \label{eq:6}
  \eear
  \ee
Here in the first, third and fourth step we apply explicit formulas for the adjunctions inside 
the respective Hom spaces, the second step is the convolution property of the Hom functor
in $\moD B$, and the fifth step uses the commutation of the Hom functor with the coend 
according to Corollary \ref{cor:oint} 
for the right exact functor ${-}\otA x \otB y^* \colon \moD A \To \moD A$.
 \\
This proves \eqref{eq:2}. The first isomorphism in \eqref{eq:2e} is a direct consequence of 
the \delt\ property \eqref{eq:codelta-property} of the dual Hom functor, while the second
isomorphism in \eqref{eq:2e} follows from \eqref{eq:2} by applying the duality functor and
noting that dualizing turns a left exact coend into a right exact end.
\end{proof}

Using these results we obtain

\begin{Lemma} \label{lemma:Del-end}
Let $\calc \eq \cala \boti \calb$ and $\calx$ be finite linear categories. For
$F\colon \calcopp \Times \calc \To \calx$ a left exact bilinear
functor whose coend exists one has
  \be
  \int^{c\in\calc}\!\! F(\ol c,c) \,\cong 
  \int^{a\in\cala}\!\! \oint^{b\in\calb}\!\! F(\ol a\boti \ol b,a\boti b) \,.
  \label{eq:coend-boxtimes}
  \ee
Analogously, if the end of a right exact bilinear functor
$G\colon \calcopp \Times \calc \To \calx$ exists, then one has 
 \be
  \int_{c\in\calc} G(\ol c,c) \,\cong 
  \int_{a\in\cala}\, \oint_{b\in\calb} G(\ol a\boti \ol b,a\boti b) \,.
  \label{eq:end-boxtimes}
  \ee
\end{Lemma}

\begin{proof}
By the \delt\ property of the Hom functor we have
 \be
  \bearll
  F(\ol x,y) \!\!&\dsty
  \cong  \int^{\ol c \in \calcopp}\!\! \Ho{\calc^\opp} {\ol c} {\ol x} \otik F(\ol c,y)
  \,\cong \int^{\ol c \in \calcopp}\!\!\!\! \int^{d \in \calc}\!\!
  \Ho{\calc^\opp} {\ol c} {\ol x} \otik \Ho{\calc} d y \oti F(\ol c,d)
  \Nxl2 &\dsty 
  \cong \int^{\ol c \in \calcopp}\!\!\! \int^{d \in \calc}\!\!
  \Ho{\calc^\opp \boti \calc} {\ol c \boti d} {\ol x \boti y} \oti F(\ol c,d)
  \Nxl2 &\dsty 
  \cong \oint^{\ol c \in \calcopp}\!\!\! \oint^{d \in \calc}\!\!
  \Ho{\calc^\opp \boti \calc} {\ol c \boti d} {\ol x \boti y} \oti F(\ol c,d) \,.
  \eear
  \ee
Here the last isomorphism holds because the respective functors are already left exact.
The isomorphism \eqref{eq:coend-boxtimes} then follows by the calculation
  \be
  \bearll\dsty
  \int^{x \in \calc}\!\! F(\ol x,x) \!\!&\dsty
  \cong \oint^{x \in \calc}\!\!\! \oint^{\ol c \in \calcopp}\!\!\!\! \oint^{d \in \calc}\!
  \Ho{\calc^\opp \boti \calc} {\ol c \boti d} {\ol x \boti x} \oti F(\ol c,d)
  \Nxl2 &\dsty
  \cong \oint^{\ol c \in \calcopp}\!\!\!\! \oint^{d \in \calc}\!\!\! \oint^{x \in \calc}\!
  \Ho{\calc^\opp \boti \calc} {\ol c \boti d} {\ol x \boti x}   \oti F(\ol c,d)
  \Nxl2 &\dsty
  \!\!\stackrel{\eqref{eq:2}} \cong\!
  \oint^{\ol c \in \calcopp}\!\!\!\! \oint^{d \in \calc}\!\!\! \oint^{a \in \cala} \!\!\!
  \oint^{b \in \calb}\!\!
  \Ho{\calcopp\boxtimes\calc} {\ol c \boti d} {\ol{a\boti b} \boti a\boti b} \oti F(\ol c,d)
  \Nxl2 &\dsty
  \cong \oint^{a \in \cala}\!\!\! \oint^{b \in \calb}\!\!\!
  \oint^{\ol c \in \calcopp}\!\!\!\! \oint^{d \in \calc}\!\!
  \Ho{\calcopp\boxtimes\calc} {\ol c \boti d} {\ol{a\boti b} \boti a\boti b} \oti F(\ol c,d) 
  \Nxl2 &\dsty
  \cong \oint^{a\in\cala}\!\!\! \oint^{b\in\calb}\!\! F(\ol a\boti \ol b,a\boti b)
  \,\cong \int^{a\in\cala}\!\!\! \oint^{b\in\calb}\!\! F(\ol a\boti \ol b,a\boti b) \,.
  \eear
  \ee
Here in the second and fourth line we invoke the Fubini theorem for left exact coends
(see \cite[Thm.\,B.2]{lyub11}).
The isomorphism \eqref{eq:end-boxtimes} is seen analogously.
\end{proof}

\begin{cor}
For $\calc \eq \cala \boti \calb$ the Deligne product of finite linear categories $\cala$ 
and $\calb$ the isomorphisms
  \be
  \int^{c \in \calc}\! \ol c \boti c
  \,\cong \int^{a \in \cala}\!\!\! \int^{b \in \calb}\! \ol a \boti \ol b \boti a \boti b
  \label{eq:coend-C}
  \ee
and
  \be
  \int_{c \in \calc} \ol c \boti c 
  \,\cong \int_{a \in \cala} \int_{b \in \calb} \ol a \boti \ol b \boti a \boti b
  \label{eq:end-C}
  \ee
hold.
\end{cor}

\begin{proof}
The functor $\boxtimes\colon \calcopp {\times}\, \calc \To \calcopp \boti \calc$ is exact,
hence both of the isomorphisms \eqref{eq:coend-boxtimes} and \eqref{eq:end-boxtimes} are
applicable. Moreover, for all finite linear categories $\calx$ and all objects $x \iN \calx$ 
we have $\oint^{b\in\calb}\! x \boti \ol b \boti b \,{\cong}\, x \,{\boxtimes} \int^{b\in\calb}
\ol b \boti b$ because in this case the left exact coend is also an ordinary coend. An analogous
statement applies to the right exact end.
\end{proof}

With the help of this corollary we arrive at the following result, in which with respect to
Lemma \ref{lemma:Del-end} the role of ends and coends is interchanged:

\begin{Lemma}\label{lemma:coend-reF}
For $\calc \eq \cala \boti \calb$ and $\calx$ finite linear categories and
$F\colon \calcopp \Times \calc \To \calx$ a left exact bilinear
functor whose end exists the isomorphism
  \be
  \int_{c\in\calc}\! F(\ol c,c) \,\cong 
  \int_{a\in\cala}\! \int_{b\in\calb}\! F(\ol a\boti \ol b,a\boti b)
  \label{eq:end4le}
  \ee
of ends holds. Analogously, if the coend of a right exact bilinear functor 
$G\colon \calcopp \Times \calc \To \calx$ exists, then there is an isomorphism
  \be
  \int^{c\in\calc}\!\! G(\ol c,c) \,\cong 
  \int^{a\in\cala}\!\!\!\! \int^{b\in\calb}\!\! G(\ol a\boti \ol b,a\boti b)
  \label{eq:coend4re}
  \ee
of coends.
\end{Lemma}

\begin{proof}
Since $F$ is left exact, it induces a left exact linear functor 
$\widehat{F}\colon \calcopp \boti \calc \To \calx$. We then have
  \be
  \bearll \dsty
  \int_{c \in \calc}\! F(\ol c, c) \!\!&\dsty \cong 
  \int_{c \in \calc}\! \widehat{F}(\ol c \boti c)
  \!\stackrel{\eqref{eq:end-Fleft}}\cong\! 
  \widehat{F}\big( \int_{c \in \calc} \ol c \boti c \big) 
  \!\stackrel{\eqref{eq:end-C}}\cong\!
  \widehat{F}\big( \int_{a \in \cala} \int_{b \in \calb} \ol a \boti \ol b \boti a \boti b \big)
  \Nxl2 &\dsty 
  \!\!\!\!\stackrel{\eqref{eq:end-Fleft}}\cong\!\!
  \int_{a \in \cala} \int_{b \in \calb}  \widehat{F}( \ol a \boti \ol b \boti a \boti b) 
  \,\cong \int_{a \in \cala} \int_{b \in \calb}\!  F( \ol a \boti \ol b,  a \boti b) \,.
  \eear
  \ee
The isomorphism \eqref{eq:coend4re} follows analogously with the help of \eqref{eq:coend-C}.
\end{proof}


\subsection{Nakayama functors for finite linear categories}\label{ssec:Nakayama}

Let $A \eq (A,\mu,\eta)$ be a \findim\ \ko-algebra; recall that we assume all modules
and bimodules to be \findim. As in Section \ref{ssec:background} we denote 
the finite \ko-linear category of \findim\ left $A$-modules by $A\Mod$ and write the Hom functor 
on $A\Mod^\opp \Times A\Mod$ as $\Ho {\!A} --$. We denote the algebra opposite to $A$ by $\Aop$.
In the sequel, whenever it is required in order for an expression to make sense
that $A$ or the dual vector space $\Awee \eq \,\Ho\ko A\ko$
has the structure of a left or right $A$-module or of an $A$-bimodule, then $A$
is to be regarded as being endowed with the regular left or right $A$-action(s),
while $\Awee$ is regarded as being endowed with the duals of these actions (i.e., for short,
as the \emph{co-regular} left, right or bi-module).

The \emph{Nakayama functor} of $A\Mod$ is the right exact endofunctor
  \be
  \pift_A := \Hox {\!A} -A
  \label{eq:piftA}
  \ee
of $A$\Mod.
We also consider its left exact cousin, the endofunctor
  \be                           
  \pifu_A := \Ho {\!\Aop} {-^* \!} A
  \label{eq:pifuA}
  \ee
of $A$\Mod.
These functors can also be described (see \cite[Lemma\,III.2.9]{ASss}:
and \cite[Prop.\,3.1]{ivanS}) by the isomorphic functors
  \be
  \pift_A \cong A^*_{} \otA {-}  \qquad {\rm and} \qquad
  \pifu_A \cong \Ho{\!A} {\Awee} - \,,
  \label{eq:piftA++}
  \ee
respectively, with $\Awee$ the co-regular $A$-bimodule.
The restrictions of these two endofunctors to the subcategories of projective and of 
injective $A$-modules, respectively, are quasi-inverse equivalences 
$A$-Projmod$\;{\stackrel\simeq\longleftrightarrow}\,A$-Injmod \cite[Sect.\,3]{ivanS}.

In this paper we are interested in finite linear categories $\calx$ as abstract categories, 
not in their particular realization as $A\Mod$ for some specific algebra $A$.
Accordingly we wish to work with a purely categorical variant of the Nakayama functor and its
left exact analogue.
Note that the identity functor $\id_\calx$ is both left and right exact, so for $\calx$ a finite 
linear category we can apply both of the functors \eqref{eq:Gammalr,Gammarl} to it. 
We can thus give

\begin{Definition}
The \emph{Nakayama functor} of a finite linear category $\calx$ is the endofunctor
  \be
  \pift_\calx := \Gammarl(\id_\calx)
  = \int^{x\in\calx}\!\!\! \Hox\calx -x \oti x~\in \Funre(\calx,\calx) \,,
  \ee
i.e.\ the image of the identity functor $\id_\calx$, seen as a left exact functor, in 
$\Funre(\calx,\calx)$.
\\
The \emph{left exact analogue of the Nakayama functor} of $\calx$ is the image 
  \be
  \pifu_\calx := \Gammalr(\id_\calx) = \int_{x\in\calx} \Ho\calx x- \oti x ~\in \Funle(\calx,\calx)
  \label{eq:pifu}
  \ee
of $\id_\calx$, seen as a right exact functor, in $\Funle(\calx,\calx)$.
\end{Definition}

Applying Lemma \ref{lem:Psile-as-end} to the identity functor,
it follows immediately that these functors satisfy
  \be
  \int^{x\in\calx}\!\! \ol x  \boti \pifu_{\!\calx}(x) \,\cong \int_{y\in\calx} \ol y \boti y
  \qquad {\rm and} \qquad
  \int_{x\in\calx}\! \ol x \boti \pift_{\!\calx}(x) \,\cong \int^{y\in\calx}\!\! \ol y \boti y
  \,.
  \ee

The following considerations will show that the chosen terminology is appropriate.
To connect the functors $\pifu_\calx$ and $\pift_\calx$ to the representation theoretic
structures occurring in \eqref{eq:piftA}\,--\,\eqref{eq:piftA++}, let us choose a \findim\ 
algebra $A$ with an equivalence $A\Mod \,{\simeq}\, \calx$. From now on we tacitly identify 
$\calx$ with $A\Mod$.
Doing so we can state

\begin{Lemma}\label{lem:Pi=A*}
The Nakayama functor of a finite linear category $\calx \,{\simeq}\, A\Mod$ satisfies
  \be
  \pift_\calx \cong ({}_AA_A)^*_{} \otA {-} \,.
  \label{eq:defPI}
  \ee
\end{Lemma}
  
\begin{proof}
The Nakayama functor $\pift \,{\equiv}\, \pift_\calx$ is right exact, hence upon the 
identification $\calx \eq A\Mod$, according to (R3) of Lemma \ref{lem:shimi7:2.6} there 
is a natural isomorphism 
  \be
  \pift \cong \pift({}_AA_A) \otA -\,.
  \label{eq:pift=piftAota-}
  \ee
Now for any $y\iN \calx$ we have
an isomorphism $\Ho A {{}_AA_A} y \,{\cong}\, y$ as left $A$-modules
and hence an iso\-mor\-phism $\Hox A {{}_AA_A} y \,{\cong}\, y^*$ of right $A$-modules.
This implies that
  \be
  \pift({}_{A}A_{A}) = \int^{y\in A\Mod}\!\! \Hox A {{}_AA_A} y \oti y
  \,\cong \int^{y\in A\Mod}\!\! y^{*} \oti y \stackrel{\eqref{eq:coend=A*}}\cong ({}_AA_A)^*_{}
  \ee
as $A$-bimodules. Combining this isomorphism with the expression 
\eqref{eq:pift=piftAota-} for $\pift$ gives \eqref{eq:defPI}.
\end{proof}

Next we note that, by construction, $\pifu_\calx$ is left exact and thus has a left adjoint,
while $\pift_\calx$ is right exact and thus has a right adjoint. Taking $G \eq \id_\calx$
in Corollary \ref{cor:GammarlGra} we have in fact

\begin{Lemma}\label{lem:pifu-adjunction}
For any  finite linear category $\calx$ the functor $\pift_\calx$ is left adjoint to 
$\pifu_\calx$.
\end{Lemma}

We finally note that the expressions for the Nakayama functors have natural generalizations 
obtained by replacing the identity functor by an arbitrary linear functor. Indeed, given a 
right exact linear functor $G\colon A\Mod \To B\Mod$, the linear functor 
$\Gammalr(G)\colon A\Mod \To B\Mod$ can, with the help of the Peter-Weyl isomorphism 
\eqref{eq:end-G}, be written as
  \be
  \bearll
  \Gammalr(G) \!\!&\dsty = \int_{m \in A\Mod} \Ho A m- \otimes G(m)
  \,\cong \HO A {\int^{m\in A\Mod}\! G(m)^* \oti m } -
  \Nxl2&
  \cong \Ho A {\big( \int_{m\in A\Mod} G(m) \otimes m^* \big)^*} - \,\cong \Ho A {G(A)^*}- \,.
  \eear
  \label{eq;Gammalr(G)}
  \ee
This provides the explicit form of the Eilenberg-Watts equivalence between right and left
exact functors that exists according to Lemma \ref{lem:lexrexbimod}(i). Note that 
the calculation in \eqref{eq;Gammalr(G)} still makes sense if $G$ is just required 
to be a linear functor. Thus the prescription for $\Gammalr$ in \eqref{eq:Gammarl,Gammalr}
in fact defines a functor from the category $\Fun(A\Mod,B\Mod)$ of \emph{all} linear
functors to $\Funle(A\Mod,B\Mod)$.
Similarly it follows from \eqref{eq:coend=G(A*)}, using the description \eqref{eq:R4expl}
of the tensor product over $A$, that 
  \be
  \bearl\dsty
  \int^{m\in A\Mod}\!\! \Hox A -m \oti F(m) 
  \,\cong \int^{m\in A\Mod}\!\! \Hox A - {m \otik F(m)^*}
  \Nxl1\hspace*{4.6em}
  \cong \Big( {\dsty\int_{m\in A\Mod}} \Ho A - {m \otik F(m)^*} \Big)^{\!*}
  \,\stackrel{\eqref{eq:coend-hom2r}}\cong\!  \Hox A - {\int_{m\in A\Mod} m \otik F(m)^*}
  \Nxl1\hspace*{4.6em}
  \cong \Hox A - {\big(\! \int^{m\in A\Mod}\!  F(m) \otik m^*\big)^*}
  \stackrel{\eqref{eq:coend=G(A*)}}\cong\! \Hox A - {F(A^*)^*}
  \stackrel{\eqref{eq:R4expl}}\cong\!  F(A^{*}) \otimes_A {-} \,.
  \eear
  \ee
This shows that 
  \be
  \Gammarl(F) \,\cong\, F(A^*) \otimes_A - 
  \ee
and that the prescription for $\Gammarl$ in \eqref{eq:Gammarl,Gammalr}
extends to a functor from $\Fun(A\Mod,B\Mod)$ to $\Funre(A\Mod,B\Mod)$.

\begin{Remark}
For $\calx$ a finite linear category, the Nakayama functors and the equivalences appearing 
in the triangle \eqref{tria:AoppB-lex-rex} can be used to endow the category 
$\calm \eq \calx^\opp \boti \calx$ with the structure of a Grothendieck-Verdier (GV)
category. Recall \cite{boDr} that a GV-structure on a monoidal
category $\calm$ amounts to an object $N \iN \calm$ and an equivalence 
$D_N\colon \calm \To \calmopp$ with a natural family of isomorphisms
$\Ho \calm  {x\oti y} N \cong  \Ho \calm {y} {D_{N}(x)}$.
 \\
In our case we use the equivalence of $\calm$ with the monoidal category 
$\Funre(\calx,\calx)$ to describe the GV-structure. The category $\Funre(\calx,\calx)$ 
has an obvious monoidal structure by composition of functors. For the object $N$ we take
the Nakayama functor $\pift_{\calx} \iN \Funre(\calx,\calx)$. As
the equivalence $D_{N}$ we take the functor that maps $G \iN \Funre(\calx,\calx)$ to the
functor $\Gammarl(G^{\ra}) \iN \Funre(\calx,\calx)^{\opp}$. Using that
$\pift_\calx \eq \Gammarl(\id_\calx)$ with $\Gammarl \eq \Phire \cir \Psile$
and invoking Corollary \ref{cor:oint}, for $G,H \iN \Funre(\calx,\calx)$ we obtain
  \be
  \begin{array}{lcl}
  \Nat_\Funre(H {\circ}\, G, \pift_{\calx}) \!\!\!\!\!\!&
  =&\!\! \Nat_\Funre\big( H {\circ}\, G, \Phire\cir\Psile(\id_\calx) \big)
  \,\cong \Ho\calm {\Psire(H {\circ}\, G)} {\Psile(\id_\calx)}
  \Nxl1&
  \stackrel{\eqref{Phile...Psire}}\cong &\!\!
  \Ho\calm {\int_{x\in\calx} \ol x \boti H {\circ}\, G(x)} {\int^{y\in\calx} \ol y \boti y}
  \Nxl2&
  \!\!\!\!\!\stackrel{\eqref{eq:oint1},\eqref{eq:oint2}}\cong \!\!\!\!\!&\dsty
  \oint^{x\in\calx}\!\!\! \oint^{y\in\calx}\!\!\!
  \Ho\calm {\ol x \boti H {\circ}\, G(x)} {\ol y \boti y}
  \Nxl2& 
  \cong &\!\! \dsty
  \oint^{x\in\calx}\!\!\! \oint^{y\in\calx}\!\!\! \Ho\calx yx \otik \Ho\calx {H {\circ}\, G(x)} y
  \eear
  \ee
and analogously (using that $H$ has a right adjoint $H^\ra$)
  \be
  \bearll
  \Nat_\Funre(G,\Gammarl(H^\ra)) \!\!&
  =\, \Nat_\Funre\big( G , \Phire\cir\Psile(H^\ra) \big)
  \,\cong \Ho\calm {\Psire(G)} {\Psile(H^\ra)} \qquad
  \Nxl2&\dsty
  \cong \oint^{x\in\calx}\!\!\! \oint^{y\in\calx}\!\!\! \Ho\calx yx \otik \Ho\calx {G(x)} {H^\ra(y)} \,.
  \eear
  \ee
When combined with the adjunction we thus arrive at a natural isomorphism
  \be
  \Nat_\Funre(H \cir G, \pift_\calx) \,\cong\, \Nat_\Funre(G,\Gammarl(H^\ra))
  \ee
for all $G,H \iN \Funre(\calx,\calx)$.
 \\
General results on GV-categories \cite[Prop.\,1.2]{mani27} imply that the equivalence
$D_{N}$ induces another monoidal structure on $\Funre(\calx,\calx)$ for which the
Nakayama functor $\pift_\calx$ is the monoidal unit. This monoidal structure can be 
recognized as the one that is obtained by identifying $\Funre(\calx,\calx)$, via the
functors $\Gammalr$ and $\Gammarl$, with the monoidal category $\Funle(\calx,\calx)$.
\end{Remark}


\subsection{Properties of Nakayama functors}

In this subsection we use the categorical formulation of the Nakayama 
functor to describe its behavior with respect to the composition with functors, to the 
Deligne product, and to taking opposites, and we characterize the finite linear categories
whose Nakayama functor is an equivalence. We also determine the Nakayama functors 
of the categories of left (right) exact endofunctors between linear categories. 
The corresponding results in the case of modules over algebras are rather evident. Still, 
the categorical formulation has advantages, for instance
when discussing the coherence natural isomorphisms in the following result. 

When the double left adjoint of a functor $F$ exists, we denote it by $ F^\lla$; analogously 
we denote the double right adjoint by $G^{\rra}$. We have

\begin{thm}\label{thm:picu.F=Fll.pifu}
(i)\, 
Let $F\colon \cala\To\calb$ be a left exact functor between finite linear categories
having a left exact left adjoint $F^\la$ whose left adjoint is again left exact. Then there is a 
natural isomorphism
  \be
 \varphi^{\rm l}_F \colon \quad \pifu_\calb \circ F \,\cong\, F^\lla \circ \pifu_\Cala 
  \label{eq:pifu.F=F.pifu}
  \ee
of functors. Analogously,
for a right exact functor $G\colon \cala\To\calb$ having a right exact right adjoint $G^{\ra}$
whose right adjoint is again right exact, there is a natural isomorphism
 \be
 \varphi^{\rm r}_G \colon \quad \pift_\calb \circ G \,\cong\, G^\rra \circ \pift_\Cala \,.
  \label{eq:pift.G=G.pifr}
  \ee
(ii)\,
The isomorphisms $\varphi_F^{\rm l}$ and $\varphi_G^{\rm r}$ are coherent in the following sense:
If $F_1\colon \cala \To \calb$ and $F_2\colon \calb \To \calc$ are left exact functors, 
then the diagram 
  \be
  \begin{tikzcd}
  \pifu_\calc \cir F_2 \circ F_1 \ar{r}{\varphi^{\rm l}_{F_2} \circ\, \id}
  \ar{d}{\varphi_{F_2 \circ F_1}^{\rm l}}
  & F_2^\lla \cir \pifu_\calb \cir F_1 \ar{d}{\id \,\circ\, \varphi_{F_1}^{\rm l}}
   \\
  (F_2 \cir F_1)^\lla \cir \pifu_\cala \ar{r}{\cong}
  & F_2^\lla \cir F_1^\lla \cir \pifu_\cala
  \end{tikzcd}
  \label{eq:coherence-comp}
  \ee
commutes, with the unnamed isomorphism in the lower row obtained from the
canonical isomorphism $(F_2 \cir F_1)^\lla \,{\cong}\, F_2^\lla \cir F_1^\lla$.
The analogous diagram for $\varphi^{r}$ commutes as well. 
\end{thm}

\begin{proof}
For $x \iN \cala$ we have
  \be
  \bearll
  \pifu_{\calb}(F(x)) \!\!&\dsty = \int_{b\in\calb} \Ho\calb b{F(x)} \oti b
  \,\cong \int_{b\in\calb} \Ho\cala {F^\la(b)}x \oti b
  \Nxl2&
  \cong \big[ \Ho\cala {\reflectbox?}x \oti ? \big] (\int_{b\in\calb} \ol{F^\la(b)} \boti b)
  \Nxl2&
  \!\!\!\stackrel{\eqref{eq:end-vs-Gr}}\cong
  \big[ \Ho\cala {\reflectbox?}x \oti ? \big] (\int_{a\in\cala} \ol a \boti F^\lla(a))
  \dsty
  \,\cong \int_{a\in\cala} \Ho\cala ax \oti  F^\lla(a)
  \Nxl2&
  \cong\,  F^\lla \big( \int_{a\in\cala} \Ho\cala ax \oti a \big) \,=\,  F^\lla(\pifu_\Cala(x))
  \eear
  \ee
naturally in $x$. Here in the second and third line we use that $\Ho\cala {\reflectbox?}x \oti ?$
is left exact and thus commutes with the end, and the first isomorphism in the third line 
is \eqref{eq:end-vs-Gr} with $G \eq F^\lla$; the last line holds because $F^{\lla}$ is 
required to be left exact.
\\[2pt]
The coherence statement with respect to the composition of the functors follows easily from 
the one in  Lemma \ref{lem:adj4(co)ends}.
\end{proof}

Since the double left adjoint of the right adjoint of a functor is just a left adjoint, we 
have directly
  
\begin{cor}\label{cor:left-right-rel}
(i)\,
Let $F\colon \cala\To\calb$ be an exact functor whose left adjoint is again left exact. 
Then there is a natural isomorphism $\pifu_\Cala \cir F^\ra \,{\cong}\, F^\la \cir \pifu_\calb$.
Similarly, if $G\colon \cala\To\calb$ is an exact functor whose right adjoint is again right
exact, then there is a natural isomorphism 
$\pift_\Cala \cir G^\la \,{\cong}\, G^\ra \cir \pift_\calb$.
\\[2pt]
(ii)\,
If in addition $\pifu_\Cala$ is an equivalence, then one can express the right adjoint 
of $F$ in terms of the left adjoint as 
  \be
   F^\ra \,\cong\, \pift_\Cala \circ F^\la \circ \pifu_\calb \,.
   \label{eq:Fra=N.Fla.N}
  \ee
And similarly, if $\pifu_\calb$ is an equivalence, then
$F^\la \,{\cong}\,\pifu_\Cala \cir F^\ra \cir \pift_\calb$.
\end{cor}

We continue with further categorical properties of the Nakayama functors. 

\begin{Proposition}\label{proposition:Nak-boxtimes-opp}  
(i)\, The Nakayama functors for the Deligne product $\cala \boti \calb$ of two finite linear
categories can be expressed in terms of the Nakayama functors for its factors as
  \be
  \pifu_{\Cala \boxtimes \calb} \cong \pifu_\Cala \boxtimes \pifu_\calb \qquad \text{and}
  \qquad \pift_{\Cala \boxtimes \calb} \cong \pift_\Cala \boxtimes \pift_\calb \,.
  \label{eq:7}
  \ee
(ii)\ The Nakayama functors of the opposite category of a finite linear category $\cala$ are
  \be
  \pifu_\Calaopp \cong (\pift_\Cala)^\opp \qquad \text{and} \qquad
  \pift_\Calaopp \cong (\pifu_\Cala)^\opp .
  \label{eq:8}
  \ee
\end{Proposition}

\begin{proof}
The claim (i) is seen by the computation 
  \be
  \bearll
  \pifu_{\Cala \boxtimes \calb} \!\!& \dsty
  \equiv \int_{x \in \cala \boxtimes \calb} \Ho{\cala \boxtimes \calb} x- \oti x
  \stackrel{\eqref{eq:end4le}}\cong\!
  \int_{a \in \cala} \int_{b \in \calb} \Ho{\cala \boxtimes \calb} {a \boti b} - \oti 
  (a \boti b)
  \Nxl2 & \dsty
  \cong \int_{a \in \cala} \Ho\cala a- \oti a \,\boxtimes \int_{b \in \calb} \Ho\calb b- \oti b
  \,=\, \pifu_\Cala \boxtimes \pifu_\calb \,.
  \eear
  \label{eq:10}
  \ee
The statement for $\pift_{\Cala \boxtimes \calb}$ follows by taking left adjoints in formula
\eqref{eq:10} which, when combined with Lemma \ref{lem:pifu-adjunction} implies that 
  \be
  \pift_{\Cala \boxtimes \calb}\cong (\pifu_{\Cala \boxtimes \calb})^\la
  \cong (\pifu_\Cala \boti \pifu_\calb )^\la
  \cong (\pifu_\Cala)^\la \boxtimes (\pifu_\calb)^\la \cong \pift_\Cala \boxtimes \pift_\calb \,.
  \ee
To show the first isomorphism in (ii), for $\ol a \iN \calaopp$ we compute
  \be
  \bearll
  \pifu_\Calaopp(\ol a) \!\!&\dsty = \int_{\ol c \in \calaopp} \Ho\calaopp c{\ol a} \oti \ol c
  \,\cong \int_{c \in \cala} \ol{ \Hox\cala {\ol a}c \oti c }
  \Nxl2 &\dsty 
  \cong \ol{ \int^{c\in\cala}\!\! \Hox\cala {\ol a}c \oti c } \,=\, (\pift_\Cala)^\opp(\ol a) \,.
  \eear
  \label{eq:11}
  \ee
The second isomorphism in (ii) follows from the first by interchanging $\cala$ with 
$\cala^{\opp}$.
\end{proof}

Under the equivalences \eqref{Phile...Psire} of the Deligne product of linear categories to 
the categories of 
left and right exact functors, the Nakayama functor has the following explicit description. 

\begin{Lemma}
\label{Lemma:Nak-endo-fun}
Let $\cala$ and $\calb$ be finite linear categories.
There are canonical isomorphisms
  \be
  \pifu_{\Funle(\cala,\calb)} (F) \,\cong\, \pifu_\calb \circ F \circ \pifu_\Cala \qquad 
  \text{for}\quad F \in \Funle(\cala,\calb) ~
  \label{eq:9}
  \ee
and
  \be
  \pift_{\Funre(\cala,\calb)} (G) \,\cong\, \pift_\calb \circ G \circ \pift_\Cala \qquad 
  \text{for}\quad G \in \Funre(\cala,\calb) \,.
  \label{eq:9a}
  \ee
\end{Lemma}

\begin{proof}
The equivalence $\Phile\colon \calaopp \boti \calb \To \Funle(\cala,\calb)$ meets the 
requirements of Theorem \ref{thm:picu.F=Fll.pifu}. It follows that there is a natural isomorphism
$\pifu_{\Funle(\cala,\calb)} \cir \Phile \,{\cong}\, \Phile \cir \pifu_{\Calaopp\boxtimes\calb}$.
By Proposition \ref{proposition:Nak-boxtimes-opp} we have $\pifu_{\Calaopp\boxtimes\calb} 
\,{\cong}\, (\pift_\Cala)^\opp \boti \pifu_\calb$.
 \\
For linear endofunctors $F$ of $\cala$ and $G$ of $\calb$ denote by $L_G$ and $R_F$
the linear endofunctors of $\Funle(\cala,\calb)$ given by post- and pre-composition
with $G$ and $F$, respectively. We compute 
  \be
  \bearll
  \Phile \circ \pifu_{\Calaopp\boxtimes\calb} (\ol a \boti b) \!\!&
  \cong  \Ho\cala {\pift_\Cala(a)}- \oti \pifu_\calb(b)
  \Nxl2 &
  \cong \Ho\cala a {\pifu_\Cala(-)} \oti \pifu_\calb(b)
  = L_{\pifu_\calb} \circ R_{\pifu_\Cala} \circ \Phile(\ol a \boti b) \,,
  \eear
  \label{eq:12}
  \ee
where the second isomorphism holds by Lemma \ref{lem:pifu-adjunction}.
This is isomorphic to $\pifu_{\Funle(\cala,\calb)} \cir \Phile(\ol{a} \boti b)$. Since 
$\Phile$ is an equivalence, we conclude that there are isomorphisms \eqref{eq:9}. 
 \\
 The isomorphisms \eqref{eq:9a} follow analogously by considering the equivalence $\Phire$.
\end{proof}


\medskip

A natural question is under which conditions the Nakayama functors $\pifu_\calx$ and
$\pift_\calx$ are equivalences. Consider first the case that $\calx$ is the category $A\Mod$
of \findim\ modules over some \findim\ algebra $A$. We have (see e.g.\ \cite[Prop.\,IV.3.1]{AUrs})

\begin{Lemma}\label{lem:self-injective}
The following properties of a \findim\ \ko-algebra $A$ are equivalent:
\\[2pt]
(i)\, The Nakayama functor \eqref{eq:piftA} and the functor \eqref{eq:pifuA} are quasi-inverse.
\\[2pt]
(ii)\, $A$ is injective as a left $A$-module, 
\\[2pt]
(iii)\, $A$ is injective as a right $A$-module, 
\\[2pt]
(iv)\, The class of (left or right) projective modules coincides with 
the class of (left or right) injective modules.
\end{Lemma}

An algebra with any (and thus all) of these properties is called \emph{self-injective}.
For self-injective $A$ the functors \eqref{eq:piftA} and \eqref{eq:pifuA} are in
particular exact and there are natural isomorphisms \cite[Thm.\,3.3]{ivanS}
  \be
  \pift_A \cong \HO{\!A} { \Ho{\!A} {A^*_{}} A} -  \qquad {\rm and} \qquad
  \pifu_A \cong \Ho{\!A} {A^*_{}} A \otimes_{\!A} - \,,
  \ee
where $\Ho{\!A} {A^*_{}} A$ is seen as a bimodule and hence 
$\HO{\!A} { \Ho{\!A} {A^*_{}} A} {{}_Am}$ as a left module.

If $A$ has a structure of Frobenius algebra, then it is in particular self-injective.
The Nakayama automorphism of a Frobenius algebra $A$ with Frobenius form $\kappa$ is the
automorphism $\nu$ of $A$ defined by the equalities 
$\kappa(\alpha,\nu(\beta)) \eq \kappa(\beta,\alpha)$ for all $\alpha,\beta \iN A$.
The functors $\pift_A$ and $\pifu_A$ amount to twisting the action of $A$ on a
module by the Nakayama automorphism $\nu$ of $A$ and by $\nu^{-1}$,
respectively. It follows that if $A$ is Frobenius, then $A^*$ and $\Ho{\!A} {A^*_{}} A$
are isomorphic as bimodules to the bimodules obtained from the regular bimodule ${}_A A_A$ 
by twisting the right and left $A$-action, respectively,
by the Nakayama automorphism \cite[Cor.\,3.4]{ivanS}.
A Frobenius algebra is symmetric iff the Nakayama automorphism is inner, and thus
iff $\Awee$ and $A$ are isomorphic as bimodules, and iff the Nakayama functor is isomorphic
to the identity functor.

The properties of being self-injective and of having the structure of a symmetric Frobenius 
algebra are Morita invariant. 
Accordingly, following a suggestion by Shimizu \cite{shimiP} we give

\begin{Definition}\label{Def:symFrobcat}
(i)\, A finite linear category is called \emph{self-injective} iff it is equivalent 
as a linear category to the category of modules over a self-injective \ko-algebra. 
\\[2pt]
(ii)\, A finite linear category is called \emph{symmetric Frobenius} iff it is equivalent 
as a linear category to the category of modules over a symmetric Frobenius \ko-algebra. 
\end{Definition}

Combining the previous statements with Lemma \ref{lem:Pi=A*} we conclude:

\begin{Proposition}\label{prop:Nakayama-equiv}
Let $\calx$ be a finite linear category.
\\[2pt]
(i)\, The Nakayama functors $\pifu_\calx$ and $\pift_\calx$ are equivalences iff 
$\calx$ is self-injective. In this case they are quasi-inverse to each other.
\\[2pt]
(ii)\, $\pifu_\calx$ and $\pift_\calx$ are isomorphic to the identity functor iff 
$\calx$ is symmetric Frobenius.
\end{Proposition}

In fact, structures of a symmetric Frobenius algebra on an algebra $A$ over a field
are in bijection with bimodule isomorphisms from $A$ to $A^\wee$. Thus
for for $\calx \,{\simeq}\, A\Mod$
there is a bijection between isomorphisms from $\pifu_\calx$ (or $\pift_\calx$)
to $\id_\calx$ and structures of a symmetric Frobenius algebra on $A$.
It is thus tempting to call, for $\calx$ a finite linear category,
an isomorphism between (say) $\pift_\calx$ and the identity functor a 
\emph{symmetric Frobenius structure} on $\calx$.


\subsection{Nakayama functors and dualities} 
\label{subsec:Nakayama-adjunction}

Taking linear categories as objects, either right or left exact linear functors between 
them as 1-morphisms, and the Deligne product as a symmetric monoidal structure, 
one obtains two monoidal 2-ca\-te\-gories $\Funre$ ($\Funle$) of linear categories, right 
(left) exact linear functors and natural transformations.
(The restriction to either left or right exact functors is required because
the Deligne product of a left and a right exact functor is not defined, in general.)
By taking isomorphism classes of functors these 2-ca\-te\-gories give two monoidal categories
$h\Funre$ and $h\Funle$.
In the sequel we show that a linear category and its opposite are dual both in $h\Funre$ 
and in $h\Funle$ and express the Nakayama functors as combinations of the corresponding 
duality structures in $\Funre$ and $\Funle$. 
Thus for any finite linear category $\calm$ we consider the pair
  \be
  \begin{array}{rcl}
  b_\calm: ~~ \vect &\!\!\!\longrightarrow\!\!\!& \calm \boti \calmopp
  \Nxl1
  v &\!\!\!\longmapsto\!\!\!& v \,\,{\otimes} \int_{m\in\calm} m \boti \ol m 
  \eear \quad{\rm and}\quad \begin{array}{rcl}
  d_\calm: ~~ \calmopp \boti \calm &\!\!\!\longrightarrow\!\!\!& \vect
  \Nxl1
  \ol m \boti m' &\!\!\!\longmapsto\!\!\!& \Hox\calm {m'}m
  \eear
  \label{eq:bMdM}
  \ee
of functors, and similarly the pair
  \be
  \begin{array}{rcl}
  \tilde b_\calm: ~~ \vect &\!\!\!\longrightarrow\!\!\!& \calmopp \boti \calm
  \Nxl1
  v &\!\!\!\longmapsto\!\!\!& v \,\,{\otimes} \int^{m\in\calm} \ol m \boti m 
  \eear \quad{\rm and}\quad \begin{array}{rcl}
  \tilde d_\calm: ~~ \calm \boti \calmopp &\!\!\!\longrightarrow\!\!\!& \vect
  \Nxl1
  m \boti \ol{m'} &\!\!\!\longmapsto\!\!\!& \Ho\calm {m'}m \,.
  \eear
  \label{eq:tildebMdM}
  \ee
Being functors from \vect\ to a linear category, $b_\calm$ and $\tilde b_\calm$ are exact,
while by definition $d_\calm$ is right exact and $\tilde d_\calm$ is left exact.
 
We can express the functors $\pifu_\calm$ and $\pift_\calm$ in terms of these functors as
follows. Denote by $\sigma_{\calm,\caln}\colon m \boti n \,{\mapsto}\, n \boti m$ the 
symmetry of the Deligne product $\calm \boti \caln \,{\cong}\, \caln \boti \calm$
and abbreviate $\sigma_\calm \,{:=}\, \sigma_{\calm,\calmopp}$. Then
  \be
  \pifu_\calm =
  (\tilde d_\calm \boti \id_\calm) \circ [ \id_\calm \boti (\sigma_\calm \cir b_\calm)]
  \label{eq:pifu-vis-evcoev}
  \ee
and
  \be
  \pift_\calm
  = (\id_\calm \boti d_\calm) \circ [ (\sigma_\calmopp \cir \tilde b_\calm) \boti \id_\calm ] \,.
  \label{eq:pift-vis-evcoev}
  \ee

\begin{Lemma}
The functor $\tilde d_\calm$ is right adjoint to $b_\calm$, and 
$d_\calm$ is left adjoint to $\tilde b_\calm$.
\end{Lemma}

\begin{proof}
By definition we have
  \be
  \tilde b_\calm \,=\, - \otimes \Psile(\id_\calm) \qquad{\rm and}\qquad
  \sigma_\calm \cir b_\calm \,\cong\, - \otimes \Psire(\id_\calm) \,.
  \ee
For any $v \iN \vect$, $x\iN\calm$ and $\ol y \iN \calmopp$ the right adjoint
$b_\calm^\ra$ satisfies
  \be
  \bearll
  v^\wee \otik b_\calm^\ra (x \boti \ol y) \!\!&
  \cong \Ho\ko v {(b_\calm^\ra \cir \sigma_\calmopp) (\ol y \boti x)}
  \cong \Ho {\calmopp\boxtimes\calm} {\sigma_\calm \cir b_\calm(v)} {\ol y \boti x}
  \Nxl1 &
  \cong v^\wee \otik \Ho {\calmopp\boxtimes\calm} {\Psire(\id_\calm)} {\ol y \boti x}
  \!\stackrel{\eqref{PsirG,ac=Ga,c}}\cong\! v^\wee \otik \Ho\calm yx \,.
  \eear
  \ee
Thus $b_\calm^\ra(x \boti \ol y) \,{\cong}\, \Ho\calm yx \eq \tilde d_\calm(x \boti \ol y)$.
Similarly, for $\tilde b_\calm^\la$ we have
  \be
  \bearll
  v \otik \big( \tilde b_\calm^\la(\ol y \boti x) \big)^\wee \!\!&
  \cong \Ho\ko {\tilde b_\calm^\la(\ol y \boti x)} v
  \Nxl1 &
  \cong v \otik \Ho {\calmopp\boxtimes\calm} {\ol y \boti x} {\Psile(\id_\calm)}
  \!\stackrel{\eqref{ac,PsilF=c,Fa}}\cong\! v \otik \Ho\calm xy
  \eear
  \ee
and thus $\tilde b_\calm^\la(\ol y \boti x) \,{\cong}\, \Hox\calm xy \eq d_\calm(\ol y \boti x)$.
\end{proof}

Similar calculations show that the other adjoints of $b$ and $\tilde b$ are given by
  \be
  \tilde b_\calm^\ra(\ol y \boti x) \,\cong \int_{m\in\calm} \Ho\calm ym \otik \Ho\calm mx
  \ee
and
  \be
  b_\calm^\la(x \boti \ol y) \,\cong\, \big(\!\int_{m\in\calm} \Ho\calm xm \otik \Ho\calm my\big)^* .
  \ee

\begin{Remark}
The functors \eqref{eq:bMdM} and \eqref{eq:tildebMdM} satisfy zigzag identities up to 
natural isomorphisms. In the case of \eqref{eq:bMdM} we have natural isomorphisms
  \be
  \begin{array}{rl}
  (\id_\calm \boti d_\calm) \circ (b_\calm \boti \id_\calm): \quad m' \,\longmapsto 
  \!\!\!&\dsty
  \int_{m\in\calm} m \oti \Hox\calm {m'}m \,\cong\, m'
  \qquad {\rm and}
  \Nxl2 \dsty
  (d_\calm \boti \id_{\calm^\opp_{}}) \circ (\id_{\calm^\opp_{}} \boti b_\calm): \quad
  \ol{m'} \,\longmapsto 
  \!\!\!&\dsty
  \int_{m\in\calm} \Hox\calm m{m'} \oti \ol m 
  \Nxl2 = &\dsty
  \!\! \int_{\ol m\in\calmopp} \Hox\calmopp {\ol{m'}} {\ol m} \oti \ol m \,\cong\, \ol{m'} \,, 
  \eear
  \label{eq:zigzag}
  \ee
where in both cases we use the \delt\ property 
(as well as \eqref{eq:end-Fleft} applied to the Deligne product, which is an exact functor).
The calculation for the functors \eqref{eq:tildebMdM} is similar.
In view of these identities one may wish to think about the functors \eqref{eq:bMdM} 
as furnishing a notion of duality for the functor category $\Funre(\calm,\caln)$
(and analogously \eqref{eq:tildebMdM} for $\Funle(\calm,\caln)$). 
Accordingly one may put
  \be
  G^* := 
  (d_\caln \boti \id_{\calm^\opp_{}}) \circ (\id_{\caln^\opp_{}} \boti G \boti \id_{\calm^\opp_{}})
  \circ (\id_{\caln^\opp_{}} \boti b_\calm): \quad \caln^\opp \To \calm^\opp 
  \label{eq:G*}
  \ee
for any $G \iN \Fun(\calm,\caln)$; this gives
$G^*(\ol n) \eq \int_{m\in\calm} \Hox\caln {G(m)}n \oti \ol m$.
If $G$ is right exact, then it has a right adjoint $G^\ra$, and by direct calculation one
concludes that
$G^* \,{\cong}\, (G^\ra)^\opp_{}$.
Note that $G^\ra$ is left exact, so its opposite ${G^\ra}^\opp$ is right exact,
hence the mapping \eqref{eq:G*} sends right exact functors to right exact ones.
Further, noting that
$(G^\ra)^\opp_{} \,{\cong}\, ( G^\opp_{})^\la$, it follows in particular that 
  \be
  G^{**} \cong \big( \big( (G^\ra)^\opp_{} \big)^\ra \big)^\opp
  \cong \big( \big( (G^\opp_{})^\la_{} \big)^\ra \big)^\opp
  \cong \big(  G^\opp_{} \big)^\opp \cong G \,.
  \label{eq:G**=G}
  \ee
\end{Remark}


\section{Eilenberg-Watts calculus for module categories}

\subsection{Module categories}

We now study the compatibility between the constructions considered above and the structure
of a module category over a finite tensor category on the finite linear category in question.
Recall that a (left) module category over a finite tensor category $\cala$ (or, for short,
an $\cala$-module), is a finite linear category $\calm \eq \acalm$ together with a bilinear
functor, exact in the first variable, from $\cala\Times\calm$ to $\calm$, which 
we call the action of $\cala$ and simply denote by a dot, as well as natural isomorphisms 
$\mu$ and $\lambda$ with components $\mu_{a,b,m}\iN \Ho\calm {(a\oti b).m} {a.(b.m)}$ and
$\lambda_m \iN \Ho\calm {\one_\Cala .m} m$ that satisfy pentagon and triangle relations
analogous to the associator and unit constraint of a monoidal category.
Right $\cala$-modules and $\cala$-$\calb$-bimodules are defined analogously.

Our conventions concerning dualities of a rigid category $\calc$ are as follows. The right dual
of an object $c$ is denoted as $c^\vee$, and the right evaluation and coevaluation are morphisms
  \be
  \evr_c \in \Ho\calc {c^\vee \otimes c}{\one} \qquad{\rm and}\qquad
  \coevr_c \in \Ho\calc \one {c \otimes c^\vee} \,,
  \ee
while the left evaluation and coevaluation are
  \be
  \evl_c \in \Ho\calc {c \otimes \Vee c}{\one} \qquad{\rm and}\qquad
  \coevl_c \in \Ho\calc \one {\Vee c \otimes c} 
  \ee
with $\Vee c$ the left dual of $c$.

In view of their importance in some of the calculations below, we state the behavior of
the action with respect to dualities explicitly:

\begin{Lemma}\label{lem:duals-Homcalm}
Let $\cala$ be a rigid monoidal category, and let $\calm \eq \acalm$ and $\caln \eq \calna$
be left and right $\cala$-modules, respectively. Then we have  natural isomorphisms
  \be
  \Ho\calm {a.m} {m'} \cong \Ho\calm m {\Vee a.m'}
  \label{eq:a.m--Veea.m'}
  \ee
for all $a\iN\cala$ and all $m,\,m'\iN\calm$, and
  \be
  \Ho\caln {n.a} {n'} \cong \Ho\caln n {n'.a^\vee}
  \label{eq:n.a--n'.avee}
  \ee
for all $a\iN\cala$ and all $n,\,n'\iN\caln$.
\end{Lemma}

\begin{proof}
This is seen in the same way as in the case of a (not necessarily strict) tensor product:
The defining properties of the left duality and the naturality of the mixed associator $\mu$
directly imply that the linear maps
  \be
  \Ho\calm {a.m} {m'} \ni~ f \,\longmapsto\, 
  (\id_{\Vee_{}a} \,.\,f) \circ \mu_{\Vee_{}a,a,m}^{} \circ (\coevl_a \,.\, \id_m)
  ~\in \Ho\calm m {\Vee a.m'}
  \ee
and
  \be
  \Ho\calm m {\Vee a.m'} \ni~ g \,\longmapsto\,
  (\evl_a \,.\, \id_{m'}) \circ \mu_{a,\Vee_{}a,m'}^{-1} \circ (\id_a \,.\, g) 
  ~\in \Ho\calm {a.m} {m'} 
  \ee
are mutually inverse, thus establishing \eqref{eq:a.m--Veea.m'}.
The validity of \eqref{eq:n.a--n'.avee} follows analogously, now using the properties
of the right duality.
\end{proof}

\begin{Remark}\label{Remark:adjoint-action}
The left action with a fixed object $a \iN \cala$ on an $\cala$-module $\acalm$ defines 
a linear en\-do\-functor 
  \be
  F_{a}: \quad m \longmapsto a\,.\,m 
  \label{eq:left-action-a}
  \ee
of $\calm$. According to Lemma \ref{lem:duals-Homcalm} the functor $F_{^{\vee\!}a}$ 
is right adjoint to $F_a$. Moreover, reading the formula 
         \eqref{eq:a.m--Veea.m'} 
backwards shows that $F_a$ also
has a left adjoint, namely $F_{a^{\vee}}$, and thus the en\-do\-functor $F_a$ is exact.
Similarly, the right action with $a$ on $\calna$ defines an exact endofunctor $G_{a}$ 
which has $G_{a^{\vee}}$ and $G_{^{\vee\!}a}$ as a right and left adjoint, respectively.
\end{Remark}

Combining the observation in this remark with Lemma \ref{lem:adj4(co)ends}
yields the following result for (co)ends taken over module categories:

\begin{cor}\label{cor:Action-boxtimes}
Let $\calm$ be a $\cala$-$\calb$-bimodule category over finite tensor categories.
Then we have
  \be
  \int_{m\in\calm} \ol m \boti a.m.b
  \,\cong \int_{m\in\calm} \ol{{}^{\vee\!}a.m.b^\vee} \boti m 
  \label{eq:Action-boxtimes}
  \ee
as well as
  \be
  \int^{m\in\calm}\! \ol m \boti a.m.b 
  \,\cong \int^{m\in\calm} \ol{{a^{\vee\!}.m.{}^{\vee\!}b}} \boti m 
  \label{eq:Action-boxtimes2}
  \ee
naturally in $a\iN\cala$ and $b\iN\calb$
and coherently with respect to the monoidal structures of $\cala$ and $\calb$.
\end{cor}


\subsection{Radford theorems for bimodule categories}

Just like for any finite linear category, the Nakayama functor can be defined for a
finite module or bimodule category $\calm$ over a finite tensor category. It is then natural
to study how the functors $\pifu_\calm$ and $\pift_\calm$ relate to
the module category structure.

For $\calm$ an $\cala$-$\calb$-bimodule category over finite tensor categories
$\cala$ and $\calb$, denote by $\rrMll$ the $\cala$-$\calb$-bimodule that coincides with
$\calm$ as a finite linear category but for which the left $\cala$-action is twisted by the 
double right dual functor and the right $\calb$-action is twisted by the double left dual 
functor, and by $\llMrr$ the one for which the left $\cala$-action is twisted by the double 
left dual functor and the right $\calb$-action by the double right dual functor.
We have

\begin{thm}\label{thm:pifuMbimodfunc}[Radford theorem for finite bimodule categories]
\\[2pt]
Let $\cala$ and $\calb$ be finite tensor categories and $\calm$ an $\cala$-$\calb$-bimodule
category. Then the endofunctor $\pifu_\calm$ of the category $\calm$ has a natural
structure of a bimodule functor
  \be
  \pifu_\calm\colon~ \calm \to \rrMll ,
  \ee
i.e.\ there are coherent isomorphisms
  \be
  \pifu_\calm(a.m.b) \cong a^{\vee\vee\!}.\,\pifu_\calm(m)\,.{}^{\!\vee\vee\!}b
  \label{eq:NlM-bimodfunctor}
  \ee
for all $m\iN\calm$, $a\iN\cala$ and $b\iN\calb$.
\end{thm}

The terminology \emph{Radford theorem} is appropriate because, as we will learn in the next
subsection (see Remark \ref{rem:Radford}), for the case of a finite tensor category seen 
as bimodule category over itself, the statement reduces to a variant of the Radford theorem 
for finite tensor category as obtained in \cite{etno2}.

\begin{proof}
The isomorphisms \eqref{eq:NlM-bimodfunctor} are obtained by applying Theorem 
\ref{thm:picu.F=Fll.pifu} to the functor \eqref{eq:left-action-a} given by left action.
\end{proof}

Analogously, considering the functor $\pift_\calm$ 
leads to the following alternative version of the Radford theorem:

\begin{thm} \label{thm:piftMbimodfunc}
For $\cala$ and $\calb$ finite tensor categories,
the Nakayama functor $\pift_\calm$ of an $\cala$-$\calb$-bi\-mo\-dule
category $\calm$ extends to a bimodule functor $\pift_\calm\colon \calm\To\llMrr$: we have
  \be
  \pift_\calm(a.m.b) \cong {}^{\vee\vee\!}a.\,\pift_\calm(m)\,.b^{\vee\vee}
  \ee
coherently for all $m\iN\calm$, $a\iN\cala$ and $b\iN\calb$.
\end{thm}
      
Next we recall \cite{etos}

\begin{Definition}
An $\cala$-module $\calm$ is called an \emph{exact} module category iff for every projective 
$p \iN \cala$ and every $m \iN \calm$ the object $p.m \iN \calm$ is projective.
\end{Definition}

Any additive module functor $\calm \To \caln$ from an exact module category $\calm$ is exact 
\cite[Prop.\,7.6.9]{EGno}; as a consequence the two previous results imply immediately
that for $\calm$ an exact module category the Nakayama 
functor $\pifu_\calm$ as well as $\pift_\calm$ are exact. Moreover, an exact module category 
satisfies the criterion (i) of Lemma \ref{lem:self-injective}
\cite[Cor.\,7.6.4]{EGno}. The lemma thus gives

\begin{cor}\label{cor:Mexact-Naka-exact}
For $\calm$ an exact module category over a finite tensor category the Nakayama functor
$\pift_\calm$ and the functor $\pifu_\calm$ are quasi-inverse equivalences 
and are thus in particular exact.
\end{cor}

\begin{Remark}
In view of Proposition \ref{prop:Nakayama-equiv} we can conclude that, for $\calm$ an exact
module category over any finite tensor category, any algebra $B \iN \vect$ such that $ B\Mod$ 
is equivalent to $\calm$ as a \ko-linear category is self-in\-ject\-ive as a \ko-algebra.
This is equivalent to the fact \cite[Cor.\,7.6.4]{EGno} that every projective object of an
exact module category over a finite (multi-)tensor category is injective and vice versa.
\end{Remark}

By considering the twisted bimodule structure from Theorem \ref{thm:pifuMbimodfunc} 
for the special case of an exact module category one obtains

\begin{cor}
Let $\calm$ and $\caln$ be module categories over a finite tensor category,
and let the $\cala$-mo\-du\-le $\calm$ be exact. Then for any module functor $F\colon \calm \To \caln$
the natural isomorphism \eqref{eq:pifu.F=F.pifu} is a module natural transformation
between module functors from $\calm$ to $\rrNll$.
\end{cor}

\begin{proof}
That we deal with a module natural isomorphism is verified by direct calculation, using that 
the action of $\cala$ is exact and hence commutes with the end and that $F$ is a module functor.
\end{proof}


\subsection{Distinguished objects}

The previous considerations remain valid if we work with (bi)module categories over 
finite \emph{multi}\,tensor categories, i.e.\ \cite[Def.\,4.1.1]{EGno} we do not need to
assume that the monoidal unit $\one$ of the tensor category $\cala$ is absolutely simple,
i.e.\ satisfies $\Ho\cala \one\one \cong \ko$. We now examine the case that the bimodule 
category under consideration is a finite multitensor category $\calm \eq \cala$ regarded as
a bimodule category over itself, which is automatically an exact bimodule category
\cite[Ex.\,7.5.5]{EGno}. We can then make use of trivial identities like
$\pifu_\cala(a) \,{\cong}\, \pifu_\cala(\one\oti a) \eq \pifu_\cala(\one . a)$.
By setting $m \eq \one$ and either $a \eq \one$ or $b\eq \one$ in
Theorems \ref{thm:pifuMbimodfunc} and \ref{thm:piftMbimodfunc} we learn that
  \be
  \bearl
  \pifu_\cala(a) \,\cong\, \pifu_\cala(\one) \oti {}^{\vee\vee\!}a
  \,\cong\, a^{\vee\vee} \oti \pifu_\cala(\one)
  \qquad{\rm and}
  \Nxl3
  \pift_\cala(a) \,\cong\, \pift_\cala(\one) \oti a^{\vee\vee} 
  \,\cong\, {}^{\vee\vee\!}a \oti \pift_\cala(\one) \,.
  \eear
  \ee
Thus we obtain

\begin{Lemma}\label{lem:pi=D...}
Let $\cala$ be a finite multitensor category. Then there are isomorphisms
  \be
  \pifu_\cala \,\cong\, D_\Cala \oti {}^{\vee\vee\!}{-} \,\cong\, {-}^{\vee\vee} \oti D_\Cala 
  \label{eq:pifu-D}
  \ee
and
  \be
  \pift_\cala \,\cong\, \widetilde D_\Cala \oti {-}^{\vee\vee}
  \,\cong\, {}^{\vee\vee\!}{-} \oti \widetilde D_\Cala
  \label{eq:pift-Dt}
  \ee
of functors, with
  \be
  D_\Cala := \pifu_\cala(\one) = \int_{a\in\cala} \Ho\cala a\one \oti a 
  \qquad{\rm and}\qquad
  \widetilde D_{\Cala} := \pift_\cala(\one) = \int^{a\in\cala}\! \Hox\cala \one a \oti a \,.
  \label{D=end..}
  \ee
\end{Lemma}

Note that $\pifu_\cala$ is by construction left exact; from the explicit expressions given in
\eqref{eq:pifu-D} and \eqref{eq:pift-Dt} we see directly that both $\pift_\cala$ and
$\pifu_\cala$ are in fact exact, in agreement with Corollary \ref{cor:Mexact-Naka-exact}
as applied to the exact bimodule category $\acalaa$.
If the category $\cala$ is clear from the context, we write just $D$ in place of $D_{\Cala}$. 

We can combine Lemmas \ref{lem:pifu-adjunction} and \ref{lem:pi=D...} and the fact that the
double dual is an equivalence to see that $ \widetilde D \,{\cong}\, {}^{\vee\vee\vee\!}D $.
Also, for the composition of the two functors $\pifu_\cala$ and $\pift_\cala$ we get
  \be
  \pifu_\cala \cir \pift_\cala
  \,\cong\, D \oti {}^{\vee\vee\!}\widetilde D \oti {-}
  \,\cong\, \widetilde D^{\vee\vee} \oti {-}^{\vee \vee \vee \vee} \oti D
  \,\cong\, \Vee D \oti {-}^{\vee \vee \vee \vee} \oti D \,,
  \label{eq:pifu.pift}
  \ee
and similar formulas for $\pift_\cala \cir \pifu_\cala$. As a consequence we obtain

\begin{Lemma}\label{lem:D-inv}
Let $\cala$ be a finite multitensor 
category. Then the object $D \iN \cala$ is invertible, with inverse $\widetilde D$.
\end{Lemma}

\begin{proof}
Since $\cala$ is exact as a bimodule category over itself, by Corollary 
\ref{cor:Mexact-Naka-exact} $\pifu_\Cala$ is an equivalence with quasi-inverse $\pift_\Cala$.
Formula \eqref{eq:pifu.pift} then implies that the endofunctor
$D \oti {}^{\vee\vee\!}\widetilde D \oti -$ of $\cala$ is isomorphic to $\id_\Cala$, hence 
the object $D \oti {}^{\vee\vee\!}\widetilde D$ is isomorphic to $\one$. Similarly,
from the natural isomorphism $\pift_\cala \cir \pifu_\cala \,{\cong}\, \id_{\Cala}$ we see
that $\widetilde D \oti D^{\vee \vee}$ is isomorphic to $\one$, and hence also 
${}^{\vee\vee\!}\widetilde D \oti D \,{\cong}\, \one$. Thus $D$ and $ {}^{\vee\vee\!}\widetilde D$
are mutual inverses. But then also $\widetilde D$ is invertible and thus satisfies
$\widetilde D^{-1} \,{\cong}\, {}^{\vee\!} \widetilde D$. It follows that
${}^{\vee\vee\!}\widetilde D \,{\cong}\, \widetilde D$, thus establishing the claim.
\end{proof}

Next note that Lemma \ref{lem:pi=D...} implies in particular that the quadruple dual 
satisfies ${-}^{\vee\vee\vee\vee} \oti D \,{\cong}\, D \oti {-}$. Thus we have

\begin{cor}\label{cor:vvvv}
The quadruple dual endofunctor of a finite tensor category
$\cala$ is naturally isomorphic to conjugation by the object $D_\Cala$:
  \be
  {-}^{\vee\vee\vee\vee} \,\cong\, D \oti {-} \oti D^{-1} .
  \label{eq:s4.4A}
  \ee
\end{cor}

\begin{Remark}\label{rem:Radford}
Comparison with \cite[Lemma\,5.1]{shimi7} shows that
the object $D_\Cala$ defined by \eqref{D=end..} coincides with the
\emph{distinguished invertible object} of $\cala$ as defined in \cite[Def.\,3.1]{etno2}.
Accordingly, Corollary \ref{cor:vvvv} is a version of the Radford $S^4$ theorem for
finite tensor categories as obtained in \cite[Thm.\,3.3]{etno2}. 
Also note that in case $\cala \eq H\Mod$ is the category of \findim\ modules over a 
\findim\ Hopf algebra $H$, the object $D_\Cala$ is the one-dimensional $H$-mo\-du\-le that is 
furnished by the distinguished group-like element of the dual Hopf algebra $H^\wee$.
\end{Remark}

Recall from  Proposition \ref{prop:Nakayama-equiv} that a finite linear category $\calc$ is 
symmetric Frobenius iff its Nakayama functor is naturally isomorphic to the 
identity. If $\calc$ is even a finite multitensor category, then by evaluating Lemma 
\ref{lem:pi=D...} at the monoidal unit $\one \iN \calc$ we see that a necessary 
condition for $\calc$ to be symmetric Frobenius is that it is unimodular. Again 
from Lemma \ref{lem:pi=D...} we thus have 
the following theorem which generalizes the result \cite[Folgerung 3.3.2]{ObSch}, see also \cite{hump2} for a related result,  that a Hopf algebra has the structure of a  symmetric Frobenius algebra if and only if it is unimodular and the square of the antipode is an inner automorphism:

\begin{thm}\cite{shimiP}\label{thm:shimiP}
A finite multitensor category $\calc$ is symmetric Frobenius if and only if it is unimodular and
the double dual functor on $\calc$ is naturally isomorphic to the identity as linear functor.
\end{thm}

In particular we have

\begin{cor}
A pivotal or spherical finite multitensor category is symmetric Frobenius
iff it is unimodular.
\end{cor} 

If $\cala$ is just a finite monoidal category, with the monoidal structure not necessarily 
a biexact functor, we can still define the objects $D_\Cala$ and $\widetilde{D}_\Cala$
by the formulas \eqref{D=end..}. They are then no longer necessarily invertible, 
but in case either left or right duals exist in $\cala$, the corresponding
isomorphisms in Lemma \ref{lem:pi=D...} are still valid. 
 
\begin{Example}\label{ex:cala=Funle(calm,calm)}
An interesting specific case is obtained when $\cala$ is the finite monoidal category
$\Funle(\calx,\calx) $ of left exact endofunctors of a finite linear category 
$\calx$. {}From Lemma \ref{Lemma:Nak-endo-fun} it follows that 
  \be
  D_{\Funle(\calx,\calx)} = \pifu_{\Funle(\calx,\calx)} (\id)
  \,\cong\, \pifu_\calx \circ  \pifu_\calx \,.
  \label{eq:D-end}
  \ee
Albeit $\Funle(\calx,\calx)$ is, in general, not a finite multitensor 
category owing to the lack of right duals, it 
follows that $D_{\Funle(\calx,\calx)}$ is invertible if $\pifu_{\calx}$ is an equivalence.
\end{Example}

\begin{Remark}
{}From Theorem \ref{thm:picu.F=Fll.pifu} we learn that for $F\colon \cala \To \calb$ an 
exact functor between finite monoidal categories
having a left exact left adjoint and satisfying $F(\one_\Cala) \eq \one_\calb$ we have
  \be
  D_\calb \cong F^\lla(D_\Cala) \,.
  \ee
If $F$ is an equivalence, so that we can take $F^\la \eq F^\ra$, this reduces to
$D_\calb \,{\cong}\, F(D_\Cala)$, thereby reproducing Corollary 3.14 of \cite{shimi8}.
\end{Remark}

Inspired by \cite{shimiP} we apply our results to monoidal functors between finite multitensor 
categories to relate their left and right adjoints and to simplify and unify proofs of 
some known results. 
First we recall a general 

\begin{Lemma}\label{Lemma:Dual-Adjoint}
Let $F\colon \cala \To \calb$ be an exact monoidal functor between finite multitensor categories. 
\\[2pt]
(i)\, There exist canonical isomorphisms
  \be
  F^\ra(b^{\vee}) \,\cong\, \big( F^\la(b) \big)^{\!\vee}
  \ee
natural in $b \iN \calb$. 
Analogously there are isomorphisms $ F^\ra({}^{\vee} b) \,{\cong}\, {}^{\vee\!}(F^\la(b))$,
and similar isomorphisms with the role of left and right adjoints interchanged.
\\[3pt]
(ii)\, $F^\la$ and $F^\ra$ commute with the double dual functors of $\calb$.
\end{Lemma}

\begin{proof}
(i) follows directly from a Yoneda type argument using that $F$, being monoidal, commutes 
with the duality functors. (ii) is obtained by applying (i) twice.
\end{proof}

We can thus state

\begin{Proposition}
Let $F\colon \cala \To \calb$ be an exact monoidal functor between finite multitensor categories 
whose left adjoint is again left exact. Then the right adjoint of $F$
can be expressed in terms of the left adjoint as
  \be
  F^\ra(b) \,\cong\, D_\Cala^{-1} \otimes F^\la(D_\calb \oti b)
  \label{eq:right-via-left}
  \ee
for $b\iN\calb$, while the double left adjoint of $F$ obeys
  \be
  F^{\lla}(a) \,\cong\, D_\calb \otimes F(D_\Cala^{-1} \oti a)
  \label{eq:doubleleft-formula}
  \ee
for $a\iN\cala$.  
\end{Proposition}

\begin{proof}
For the first statement we invoke with Corollary \ref{cor:left-right-rel} to compute
  \be
  \bearll
  F^{\ra}(b) \!\!\!&
  \stackrel{\eqref{eq:Fra=N.Fla.N}} \cong \pift_\cala(F^\la(\pifu_\calb(b)))
  \stackrel{\eqref{eq:pifu-D}} \cong \pift_\cala(F^\la(D_\calb\oti {}^{\vee\vee}b))
  \cong \pift_\cala (F^\la({}^{\vee\vee}(D_\calb \oti b)))
  \Nxl4 &
  \hspace*{5.8pt}\cong\, \pift_\cala \big( {}^{\vee\vee}(F^\la(D_\calb \oti b)) \big)
  \,\cong\, D_\Cala^{-1} \otimes F^\la(D_\calb \oti b) \,.
  \eear
  \ee
Here step three holds because ${}^{\vee \vee\!}D_\calb \,{\cong}\, D_\calb$ by 
invertibility of $D_\calb$, and Lemma \ref{Lemma:Dual-Adjoint} is used in the fourth step.
The isomorphism \eqref{eq:doubleleft-formula} follows analogously.
\end{proof}

When evaluated at the monoidal unit of a finite tensor category, the equality
\eqref{eq:doubleleft-formula} reproduces \cite[Thm.\,4.8]{shimi8}. Also, if we apply
\eqref{eq:right-via-left} to the forgetful functor $U\colon \mathcal{Z}(\cala) \To \cala$
from the Drinfeld center of a finite tensor category, then using that $\mathcal{Z}(\cala)$ is
factorizable \cite[Prop.\,4.4]{etno2} and thus unimodular, we obtain Lemma 4.7 of \cite{shimi7}. 


\subsection{Inner Hom and relative Serre functors}\label{subsec:InnerHom}

Recall (see e.g.\ \cite{reVa3}) that a right Serre functor on a linear Hom-finite additive 
category $\calc$ is an additive endofunctor $G$ together with a natural family of isomorphisms
$\Ho\calc cd \,{\xrightarrow{~\cong~}}\, \Hox\calc d{G(c)}$; left Serre functors are
defined analogously.
Right Serre functors are fully faithful; if a right Serre functor is even an equivalence,
then it is called a \emph{Serre functor}. The category $\calc$ has a Serre functor iff the 
dual Hom functors $\Hox\calc c-$ and $\Hox\calc -c$ are representable for every $c \iN \calc$.
For a finite linear category $\calc$ it follows immediately from Corollary 
\ref{cor:leftexact-etc} that the existence of a Serre functor is equivalent to semisimplicity.

\begin{Remark}
Assume for a moment that a finite linear category 
$\calm$ admits left and right Serre functors $\Sl_\calm$ and $\Sr_\calm$. 
Then they are equivalences and may be taken to form an adjoint pair. The calculation 
  \be
  \bearll\dsty
  \pift_\calm(x) \,= \int^{y\in\calm}\!\! \Hox\calm xy \oti y
  \!\! &\dsty
  \cong \int^{y\in\calm}\!\! \Ho\calm {\Sl_\calm(y)}x \oti y
  \Nxl1 &\dsty
  \cong \int^{y\in\calm}\!\! \Ho\calm y {(\Sl_\calm)^\ra(x)} \oti y 
  \stackrel{\eqref{eq:delta-property}}\cong {(\Sl_\calm)^\ra(x)} \,,
  \eear
  \ee
where the first equality is the definition of the Nakayama functor
$\pift_\calm$ and the second the definition of the Serre functor $\Sl_\calm$, then shows 
that the Nakayama functor $\pift_\calm$ of $\calm$ is isomorphic to $\Sr_\calm$, and 
analogously one sees that that $\pifu_\calm$ is isomorphic to $\Sl_\calm$.
Now as noted above $\calm$ in fact does \emph{not} possess left and right Serre functors 
unless it is semisimple. Still, by this observation it is tempting to think of the Nakayama
functors $\pifu_\calm$ and $\pift_\calm$ as substitutes for left and right Serre functors
of $\calm$.  This fits with the description of the Nakayama functors in Equations
\eqref{eq:pifu-vis-evcoev}
and \eqref{eq:pift-vis-evcoev}, which look familiar from the relation between Serre functors
and dualizability structure. One might be tempted to conclude 
that the Nakayama functors always provide Serre functors. Recall, however, that
the Deligne product of a left exact and a right exact functor is not, in general, defined; 
as a consequence, the right hand sides of \eqref{eq:pifu-vis-evcoev}
and \eqref{eq:pift-vis-evcoev} do not, in general, constitute Serre functors.
\end{Remark}

   \begin{Remark}
   An action of the Serre automorphism on the core of fully-dualizable objects of 
   (a skeletal version of) the two-dimensional framed bordism bicategory realizes
   \cite{david,heVa2} an $\mathrm{SO}(2)$-action 
       whose homotopy fixed points describe oriented topological field theories.
   This appears to provide a geometric description of the Nakayama functor in the semisimple
   case.  But our results, obtained in a purely algebraic setting, do not depend on any
   semisimplicity assumption.
   \end{Remark}

In the rest of this subsection $\cala$ is a finite tensor category.
Recall that for a left $\cala$-mo\-du\-le $\calm \eq \acalm$ the inner Hom $\IHom(-,-)$ is 
defined via a family of natural isomorphisms $\Hom_{\cala}(a,\IHom(m,n)) \,{\cong}\, 
\Ho\calm {a.m}n$ for $m,n\iN\calm$ and $a \iN \cala$. This naturally defines a left exact 
functor $\IHom(-,-)\colon \calmopp \Times \calm \To \cala$. The module category $\calm$ is an
exact $\cala$-module if and only if the functor $\IHom(-,-)$ is also right exact
\cite[Cor.\,3.15\,\&\,Prop.\,3.16]{etos}.  
  
The notion of inner Hom allows us to introduce a relative version of (left or right)
Serre functor, as an adaptation of \cite[Def.\,4.29]{schaum5}:

\begin{Definition}
Let $\calm$ be a left $\cala$-module. A \emph{right relative Serre functor} on $\calm$
is an en\-do\-functor $\Sr_\calm$ of $\calm$ together with a family 
  \be
  \IHom(m,n)^{\vee} \xrightarrow{~\cong~} \IHom(n,\Sr_\calm(m)) 
  \label{eq:Serref}
  \ee
of isomorphisms natural in $m,n\iN\calm$.
Analogously, a \emph{left relative Serre functor} $\Sl_\calm$ comes with a family
  \be
  {}^{\vee}\IHom(m,n) \xrightarrow{~\cong~} \IHom(\Sl_\calm(n),m)
  \label{eq:leftSerre}
  \ee
of natural isomorphisms.
\end{Definition}

\begin{Lemma}
A right relative Serre functor on $\calm$ is a twisted module functor in the sense that there 
are coherent natural isomorphisms 
  \be
  \Sr_\calm(a.m) \,\cong\, a^{\vee\vee} .\, \Sr_\calm(m) \,.
  \label{eq:Sr(am)=avv.Sr(m)}
  \ee
Similarly there are coherent natural isomorphisms 
$\Sl_\calm(a.m) \,{\cong}\, {}^{\vee\vee}a \,.\, \Sl_{\calm}(m) $.
\end{Lemma}

\begin{proof}
We have a chain
  \be
  \bearll 
  \IHom(n,\Sr_\calm(a.m)) \!\!& \xrightarrow{~\cong~} \IHom(a.m,n)^{\vee}
  \xrightarrow{~\cong~} a^{\vee\vee} \otimes \IHom(m,n)^\vee
  \Nxl2 &
  \xrightarrow{~\cong~} a^{\vee \vee} \otimes \IHom(n,\Sr_\calm(m))
  \xrightarrow{~\cong~} \IHom(n,a^{\vee \vee}.\,\Sr_\calm(m))
  \label{eq:right-Serre}
  \eear
  \ee
of isomorphisms which is natural in $m,n  \iN \calm$ and in $a \iN \cala$. By the enriched Yoneda 
lemma (see e.g.\ \cite[Lemma\,4.11]{schaum5} for an adaption to module categories) this induces 
the natural isomorphisms \ref{eq:Sr(am)=avv.Sr(m)}, and these are by construction coherent 
with respect to the monoidal structure. The statement for $\Sl_{\calm}$ is derived analogously. 
\end{proof}

It follows directly from the definitions that if $\calma$ has both a left and a right 
relative Serre functor, then these are quasi-inverse to each other. 
Since the inner Hom functor is left exact in both arguments, the functors 
$\IHom(-,-)^{\vee}$ and ${}^{\vee}\IHom(-,-)$ from $\calm \Times \calmopp$ to $\cala$
are right exact in both arguments. It follows that if 
a module category has a left and a right Serre functor, then the inner Hom functor is exact 
so that, by the proof of Proposition 3.13 of \cite{etos}, $\calma$ is exact. 
Conversely, if $\calma$ is an exact module category, then the 
functors $\IHom(-,-)^{\vee}$ and $ {}^{\vee}\IHom(-,-)$
are exact and hence, by the enriched Yoneda lemma, representable. Thus an exact 
module category has both relative Serre functors. We summarize these findings in

\begin{Proposition}\label{prop:exirelSerre}
An $\cala$-module $\calma$ has a left and a right relative Serre functor if and only if 
it is exact as a module category over $\cala$.
\end{Proposition}

Similarly as the Hom functor, the inner Hom functor $\IHom$ is left exact in each argument, while 
${}^{\vee}\IHom$ is right exact. The respective adjoint functors can be given explicitly. To
see this, note that every object $m \iN \calm$ furnishes a module functor
  \be
  \begin{array}{lr}
  F_m\colon & {}_\cala\cala \longrightarrow \acalm~ \Nxl1 &
  a \longmapsto a.m \,. 
  \eear
  \label{eq:defFm}
  \ee
We have

\begin{Lemma}\label{lem:Fm-adj}
For $\calm$ an exact module category over a finite tensor category $\cala$ and $m \iN \calm$,
the functor $\IHom(m,-)$ is right adjoint to the functor $F_m$,
and ${}^\vee\IHom(-,m)$ is left adjoint to $F_m$.
\end{Lemma}

\begin{proof}
The first of the claims is shown by the calculation
  \be
  \Ho\cala a {F_m^\ra(n)} \cong \Ho\calm {F_m(a)} n \cong \Ho\cala a {\IHom(m,n)} \,,
  \ee
combined with the Yoneda lemma. Similarly, after noticing that
  \be
  \Ho\cala {{}^\vee\IHom(n,m)} a \cong \Ho\cala {a^\vee} {\IHom(n,m)}
  \cong \Ho\calm {a^\vee\!.n} m
  \stackrel{\eqref{eq:a.m--Veea.m'}}\cong \Ho\calm n {a.m} \,,
  \ee
the second claim follows by
$\Ho\cala {F_m^\la(n)} a \,{\cong}\, \Ho\calm n {F_m(a)} \,{\cong}\, 
\Ho\cala {{}^\vee\IHom(n,m)} a$.
\end{proof}

If we regard the finite tensor category $\cala$ as an (exact) module category over itself, 
then the relative Serre functors are given by the right and left double dual functor, 
respectively. Accordingly we can regard the relative Serre functors as generalizations of 
the double dual functors. Then, recalling from Corollary \ref{cor:Mexact-Naka-exact} that 
for an exact module category the Nakayama functors are equivalences, we obtain the following 
result which generalizes Lemma \ref{lem:pi=D...} from the regular left module category
${}_\cala \cala$ to arbitrary exact module categories.

\begin{thm}\label{thm:N=D.S}
Let $\calm$ be an exact module category over a finite tensor category $\cala$. Then there 
are isomorphisms
  \be
  \pifu_\calm \,\cong\, D\,.\,\Sl_\calm
  \label{eq:pi-ex}
  \ee
and
  \be
  \pift_\calm \,\cong\, D^{-1}.\,\Sr_\calm
  \label{eq:pi-exr}
  \ee
of module endofunctors, with $D$ the distinguished object of $\cala$.
\end{thm}

\begin{proof}
According to Theorem \ref{thm:picu.F=Fll.pifu}, for any exact $\cala$-module $\caln$ and
any module functor $F\colon \caln \To \calm$ there is an equivalence 
$ F^\lla \cir \pifu_{\calm} \,{\cong}\, \pifu_\caln\cir F$.  Further,
recalling from Lemma \ref{lem:Fm-adj} that the functor
$F_m$ introduced in \eqref{eq:defFm} satisfies $F_m^\la \,{\cong}\, {}^\vee\IHom(-,m)$, we can calculate
  \be
  \Ho\cala a {F_m^\la(n)} \,\cong \Ho\cala a {{}^{\vee}\IHom(n,m)} 
  \stackrel{\eqref{eq:leftSerre}}\cong\! \Ho\cala {a} {\IHom(\Sl_\calm(m),n)}
  \,\cong\, \Ho\cala {a .\, \Sl_{\calm}(m)} n \,.
  \ee
By the module version of the Yoneda lemma \cite[Lemma\,4.11\,\&\,Prop.\,4.12]{schaum5},
this shows that we have $F^\lla_m(-) \,{\cong}\, {-} \,.\, \Sl_\calm(m)$ as module functors. 
Using the isomorphism $ F^\lla_{m} \cir \pifu_{\cala} \,{\cong}\, \pifu_\calm\cir F_m$,
evaluated at the monoidal unit $\one \iN \cala$, together with the expression
\eqref{eq:pifu-D} for $\pifu_\cala$ we obtain the stated expression for $\pifu_\calm$.
Analogously, after invoking $F_m^\ra \,{\cong}\, \IHom(m,-)$ from Lemma \ref{lem:Fm-adj} 
we can show that $F^\rra_m(a) \eq a \,.\, \Sr_\calm(m)$ which, in turn, together with
Theorem \ref{thm:picu.F=Fll.pifu} and the expression \eqref{eq:pift-Dt} for $\pift_\cala$
implies \eqref{eq:pi-exr}.
\end{proof}

\begin{Remark}
A finite tensor category $\cala$ is called unimodular iff its distinguished invertible 
object is isomorphic to the monoidal unit \cite[Def.\,4.1]{etno2}. This generalizes the 
Hopf algebraic notion of unimodularity: if $\cala \eq H\Mod$ is the category of \findim\
modules over a \findim\ Hopf algebra $H$, then $\cala$ is unimodular iff $H$ is unimodular,
i.e.\ has a two-sided integral. Now note that for a unimodular finite tensor category $\cala$ 
the Nakayama functors are isomorphic to the (left or right) double dual, and the quadruple dual 
of $\cala$ is trivial. In view of Theorem \ref{thm:N=D.S} it is thus tempting to extend the
notion in the following way further to exact module categories:
We call an exact module category $\acalm$ over a (not necessarily unimodular) 
finite tensor category $\cala$ \emph{unimodular}
iff there exists a module natural isomorphism $\pifu_{\calm} \,{\cong}\, \Sr_{\calm}$
between the Nakayama functor and the right relative Serre functor of $\acalm$.
 \\
This indeed generalizes the notion of unimodular tensor category: For a finite tensor category 
$\cala$, the right relative Serre functor of the regular $\cala$-module ${}_{\cala}\cala$ is 
the double right dual functor. Thus
by the characterization of $\pifu_{\cala}$ in Lemma \ref{lem:pi=D...}, ${}_{\cala}\cala$ 
is unimodular iff there exists an isomorphism $D_{\cala} \,{\cong}\, \one$.
 \\
Note that any two module natural isomorphisms $\pifu_{\calm} \,{\cong}\, \Sr_{\calm}$ 
differ by a module natural automorphism of $\Sr_{\calm}$. 
Now every exact module category can be written as the direct sum of indecomposable ones
\cite{etos}. If $\calm$ is an indecomposable module category, then it follows from the 
invertibility of $\Sr_{\calm}$ that $\Fun_\Cala(\Sr_{\calm},\Sr_{\calm}) \,{\cong}\, \ko$
is one-dimensional, so that in this case
the module natural isomorphism $\pifu_{\calm} \cong \Sr_{\calm}$ is unique up to a scalar. 
 \\
If the finite tensor category $\cala$ itself is unimodular, it follows that an exact module 
category  $\acalm$ is unimodular iff there exists a module natural isomorphism 
$\Sl \,{\cong}\, \Sr$ between its relative Serre functors. 
\end{Remark} 

\vskip 3em

\noindent
{\sc Acknowledgements:}\\[.3em]
We thank Kenichi Shimizu for a helpful correspondence,
  and Manuel B\"arenz and Stefan Zetsche for comments on the manuscript.
JF is supported by VR under project no.\ 621-2013-4207.
CS is partially supported by the Collaborative Research Centre 676 ``Particles,
Strings and the Early Universe - the Structure of Matter and Space-Time'' and by the RTG 1670
``Mathematics inspired by String theory and Quantum Field Theory''. GS is supported by 
Nils Carqueville's project P 27513-N27 of the Austrian Science Fund.

\newpage

\newcommand\wb{\,\linebreak[0]} \def\wB {$\,$\wb}
\newcommand\Bi[2]    {\bibitem[#2]{#1}}
\newcommand\inBo[8]  {{\em #8}, in:\ {\em #1}, {#2}\ ({#3}, {#4} {#5}), p.\ {#6--#7} }
\newcommand\inBO[9]  {{\em #9}, in:\ {\em #1}, {#2}\ ({#3}, {#4} {#5}), p.\ {#6--#7} {\tt [#8]}}
\newcommand\J[7]     {{\em #7}, {#1} {#2} ({#3}) {#4--#5} {{\tt [#6]}}}
\newcommand\JO[6]    {{\em #6}, {#1} {#2} ({#3}) {#4--#5} }
\newcommand\JP[7]    {{\em #7}, {#1} ({#3}) {{\tt [#6]}}}
\newcommand\BOOK[4]  {{\em #1\/} ({#2}, {#3} {#4})}
\newcommand\PhD[2]   {{\em #2}, Ph.D.\ thesis #1}
\newcommand\Prep[2]  {{\em #2}, preprint {\tt #1}}
\def\adma  {Adv.\wb Math.}
\def\ajse  {Arabian Journal for Science and Engineering}
\def\alrt  {Algebr.\wb Represent.\wB Theory}         
\def\apcs  {Applied\wB Cate\-go\-rical\wB Struc\-tures}
\def\coma  {Con\-temp.\wb Math.}
\def\imrn  {Int.\wb Math.\wb Res.\wb Notices}
\def\izma  {Izvestiya: Math.}
\def\jims  {J.\wb Indian\wb Math.\wb Soc.}
\def\jams  {J.\wb Amer.\wb Math.\wb Soc.}
\def\joal  {J.\wB Al\-ge\-bra}
\def\joms  {J.\wb Math.\wb Sci.}
\def\jktr  {J.\wB Knot\wB Theory\wB and\wB its\wB Ramif.}
\def\mama  {ma\-nu\-scripta\wB mathematica\wb}
\def\momj  {Mos\-cow\wB Math.\wb J.}
\def\pams  {Proc.\wb Amer.\wb Math.\wb Soc.}
\def\quto  {Quantum Topology}
\def\taac  {Theo\-ry\wB and\wB Appl.\wb Cat.}

\end{document}